\documentclass[11pt,final]{article}
\usepackage{amsthm,amsfonts,graphicx,tikz,enumerate,psfrag,amssymb,amsmath}

\usepackage{fullpage}

\usepackage[utf8]{inputenc} 
\usepackage{bm}

\usepackage[colorlinks = true,
            linkcolor = blue,
            urlcolor  = blue,
            citecolor = blue,
            anchorcolor = blue]{hyperref}

\newtheorem{lemma}{Lemma}

\newtheorem{proposition}[lemma]{Proposition}
\newtheorem{definition}[lemma]{Definition}
\newtheorem{theorem}[lemma]{Theorem}

\newtheorem*{example}{Example}
\newtheorem{corollary}[lemma]{Corollary}
\newtheorem*{remark}{Remark}

\newcommand{\EE}{{\mathbf{E}}}

\newcommand{\tr}{{\rm tr}}

\newcommand{\dE}{\mathbf {E}}
\newcommand{\dP}{\mathbb{P}}

\newcommand{\dR}{\mathbb {R}}
\newcommand{\dC}{\mathbb {C}}

\newcommand{\cP}{\mathcal {P}}

\newcommand{\cI}{\mathcal{I}}

\newcommand{\si}{\sigma}

\newcommand{\SPAN}{ \mathrm{span}}

\newcommand{\NRM}[1]{{{\left\| #1\right\|}}} 
\newcommand{\PAR}[1]{{{\left(#1\right)}}} 
\newcommand{\BRA}[1]{{{\left\{#1\right\}}}} 

\newcommand{\INT}[1]{{{\left[\hspace{-1.4pt}\left[#1\right]\hspace{-1.4pt}\right]}}} 

\newcommand{\1}{1\!\!{\sf I}}\newcommand{\IND}{\1}
\newcommand{\veps}{\varepsilon}

\renewcommand{\Im}{\mathrm{Im}}
\newcommand{\AND}{\quad \mathrm{ and } \quad}

\newcommand{\hull}{\mathrm{hull}}

\newcommand{\WgU}{\mathrm{Wg}}
\newcommand{\Wg}{\mathrm{Wg}}

\newcommand{\cE}{{\mathcal{E}}}

\newcommand{\Fd}{{\mathbb{F}_d}}
\newcommand{\Un}{\mathbb{U}_n}
\newcommand{\Unq}{\mathbb{U}_{n^q}}
\newcommand{\On}{\mathbb{O}_n}

\title{Strong asymptotic freeness for 
independent uniform variables on compact groups associated to nontrivial representations}
\author{Charles Bordenave and Beno\^\i{}t Collins}

\begin{document}
\maketitle

\begin{abstract}
Voiculescu discovered asymptotic freeness of independent Haar-distributed unitary matrices. Many refinements have been obtained, including strong asymptotic freeness
of random unitaries and strong asymptotic freeness of random permutations
acting on the orthogonal of the Perron-Frobenius eigenvector.
In this paper, we consider a new matrix unitary model appearing naturally from 
representation theory of compact groups.
We fix a nontrivial signature $\rho$, i.e., two finite sequences of non-increasing natural numbers, and
for $n$ large enough, consider the irreducible representation
$V_{n,\rho}$ of $\Un$ associated with the signature $\rho$.
We consider the quotient $\mathbb{U}_{n,\rho}$ of $\Un$ viewed as a matrix subgroup of $\mathbb{U}(V_{n,\rho})$, 
and show that 
strong asymptotic freeness holds in this generalized context when drawing independent copies of the Haar measure.
We also obtain the orthogonal variant of this result. 
Thanks to classical results in representation theory, this result is closely related to strong asymptotic freeness for tensors, 
which we establish as a preliminary. 
To achieve this result, we need to develop four new tools, each of independent theoretical interest:
(i) a centered Weingarten calculus and uniform estimates thereof,
(ii) a systematic and uniform comparison of Gaussian moments and unitary moments of matrices,
(iii) a generalized and simplified operator-valued non-backtracking theory in a general $C^*$-algebra, and finally,
(iv) combinatorics of tensor moment matrices.
\end{abstract}

\section{Introduction}

Asymptotic freeness in the large dimension limit of independent GUE random matrices is a fundamental phenomenon discovered by Voiculescu 
in the nineties \cite{MR1094052}. It allows us to 
understand the spectrum of non-commuting polynomials in independent 
random matrices in the limit of large dimensions and under very general assumptions. 
Subsequently, Voiculescu proved that independent Haar unitary random matrices are almost surely asymptotically free as the dimension goes to infinity
\cite{MR1601878}.

While the above result addresses unitary groups, i.e.~the group of symmetries of a Euclidian space, it is natural to consider other groups
of symmetries, and in particular, of a finite set, i.e., symmetric groups. Here, the counterpart of Voiculescu's results was solved  by 
Nica in \cite{MR1197059}.
Asymptotic freeness for random permutations notably provides fine spectral information for operators acting on a random Schreier 
graph or random coverings of a fixed graph. It has also become folklore that some specific non-random permutations' asymptotic freeness can be achieved easily by explicitly exhibiting a sequence of finite quotients of a free group
whose kernels intersect with its trivial subgroup. 

As soon as compact matrix groups are involved,  the joint asymptotic behavior of independent random variables can legitimately be expected 
to depend on which representation of the group is being considered, and 
for example, although the above results are conclusive examples of asymptotic freeness, irrespective of the representation
in the case of the symmetric group, we are not aware of any complete result in this direction for unitary or orthogonal groups,
and the initial results deal instead with the very particular case of the fundamental representation. 
In a recent development by 
the second author, together with Gaudreau Lamarre and Male, addressed part of the problem
in \cite{MR3573218}, when a signature is fixed. 

However, although asymptotic freeness describes the macroscopic spectrum of non-commuting polynomials in the generators in large dimension efficiently,
it is not enough to analyze the existence of eigenvalues away from the limiting spectrum. 
The absence of such eigenvalues (also known as outliers) is a difficult problem whose positive answer corresponds to strong asymptotic freeness.
We refer to Section \ref{sec:main}
for a formal definition of strong asymptotic freeness.

The first breakthrough in this direction was achieved in \cite{MR2183281}
where the authors proved strong asymptotic freeness of independent GUE matrices, and it was followed by many improvements. A notable progress was 
strong asymptotic freeness of Haar unitaries in \cite{MR3205602}.
 Later, both authors in \cite{MR4024563} showed that strong asymptotic freeness holds 
for random independent permutations (i.e., Haar unitaries on the symmetric group) when viewed as $n\times n$ matrices.

To continue the parallel between asymptotic freeness and strong asymptotic freeness,
let us point out here that all results on strong asymptotic freeness involving independent copies of Haar measure on groups are basically 
only valid with respect to the
fundamental representation --- with the notable exception of   \cite{MR4024563}, where we show that we can also extend the result
to the tensor product of two fundamental representations. 
For particular operators, though, partial results have also been obtained in the context of quantum expanders, 
see \cite{MR3226740,MR2553116,MR2453786,MR2486279}.
In other words, there is no hint at the fact that the same strong asymptotic freeness would hold if one were considering
all representations simultaneously --- or equivalently, for our purposes, the left regular representation. 

This seems, however, to be a very natural question, as there exist sequences of permutations
(such as those of the  Ramanujan graphs, \cite{MR0484767,MR963118})
for which a simple polynomial --- the sum of generators --- is known to behave well for all nontrivial irreducible representations. 
This prompts us to state a question:
given an integer $n$, consider a nontrivial signature $\rho$, and the irreducible representation 
$V_{n,\rho}$ of $\Un$ associated with the signature. 
We consider the quotient $\mathbb{U}_{n,\rho}$ of $\Un$ viewed as a matrix subgroup of $\mathbb{U}(V_{n,\rho})$. 
In particular, the Haar measure on $\mathbb{U}_{n,\rho}$ has `less randomness' than the Haar measure on the full unitary group 
$\mathbb{U}(V_{n,\rho})$, 
in the sense that the group is of a much smaller dimension --- say, as a manifold. Then the question is: \emph{for which sequence of pairs 
$(n,\rho )$ is it true that $d$ independent Haar random variables of 
$\mathbb{U}_{n,\rho}$ are strongly asymptotically free?}

As a trivial observation, $\dim (V_{n,\rho})$ must tend to infinity because freeness does not occur in finite dimension, so
either $n$ or $\rho$ must tend to infinity. Fixing $n$ and letting $\rho$ tend to infinity seems to be a fascinating
problem to which we do not have an answer. 
Our main theorem is an answer to this theorem when $\rho$ is fixed, and $n$ tends to infinity. Our result is that
strong asymptotic freeness holds as soon as $\rho$ is nontrivial. It can be made uniform, provided
that $\rho$ varies slowly as a function of $n$.
The precise statement can be found in Corollary \ref{cor:main}.

The above corollary is a consequence of our main theorem, Theorem \ref{th:main}. This result establishes strong asymptotic freeness for 
tensor products of the fundamental representation 
and of the contragredient representation on the orthogonal of fixed points. 
We believe the main theorem to be of independent interest in mathematical physics and quantum information theory, 
but for the purpose of this introduction, let us
note that, as explained in Section \ref{sec:main}, thanks to elementary results in representation theory and a few classical facts of operator algebras, 
Corollary \ref{cor:main} and Theorem \ref{th:main} are equivalent. 

Our main result, Theorem \ref{th:main}, can be interpreted as a $0$-$1$ law over natural sequences of irreducible representations: 
either it is trivial (if the sequence of representations is one-dimensional) or strongly asymptotically free (in all other cases). 
In other words, this quantifies very precisely the fact that the sole known obstructions to strong asymptotic freeness
are the fixed points of a representation. Hence, there does not seem to be intermediate behavior between
strong asymptotic freeness and triviality. 

Let us digress a bit to hint at the fact that strong asymptotic freeness can naturally be expected to be much harder to achieve than
plain asymptotic freeness.
To be more specific, let us try to investigate whether the operator norm of a non-commutative linear function of independent copies of 
Haar-distributed variables could be bounded above uniformly (which must happen in the case of strong asymptotic freeness
 thanks to the Haagerup inequality). 
A naive attempt could consist of a combination of a net argument and a union bound --- and this attempt would succeed 
in the case of the context of the fundamental representation. 
 For this purpose, for $\veps, \eta >0$, let us observe that the size of an $\eta$-net of vectors in the unit ball of $V_{n,\rho}$
is of the order of $(C/\eta)^{\dim(V_{n,\rho}) }$, whereas for any $1$-Lipschitz function, the likelihood of
being $\varepsilon$-away from the median is of order $\exp (-c\varepsilon^2 n)$; 
this follows from Gromov's comparison theorem, see \cite[Theorem 4.4.27]{MR2760897}.
The involved exponential speeds $n$ and $\dim(V_{n,\rho})$ are comparable as $n$ grows iff 
$V_{n,\rho}$ is the fundamental representation or its contragradient. 
In all other cases $\dim (V_{n,\rho})\gg n$ and it is impossible, given $\eta$, to fix $\varepsilon$ such that
$(C/\eta)^{\dim(V_{n,\rho}) }\cdot \exp (-c\varepsilon^2 n)$ remains small uniformly over the dimension. 
Hence, we might expect that `soft' geometric techniques such as those exposed in the monograph \cite{MR3699754} 
are not of use here and make it less obvious
that strong asymptotic must hold. 

Likewise, analytic proofs extending  \cite{MR3205602} and \cite{MR2183281} are not handy in this case because of the tensor structure
of the objects -- there is no prospect for an easy folding trick like in \cite{MR3205602} or 
a  Schwinger-Dyson equation like in \cite{MR2183281} because of the tensor structure. 
Because of this lack of geometric or analytic methods, 
a natural approach turns out to be based on moments, which is the core technology of this paper.
It is implemented with the help of an operator-valued non-backtracking theory.
The non-backtracking theory was observed by a few authors
to be very powerful in studying outliers in the context of graph theory,  see, e.g., \cite{Friedman1996,bordenaveCAT}.  
A version of the non-backtracking theory with non-commuting weights was developed further in \cite{MR4024563}
to prove strong asymptotic freeness of random permutations. 
Here, to achieve our goal, we generalize this theory even further, beyond the case of permutation operators, to the case of all bounded operators. 
The results are gathered in Section \ref{sec:SAFNB}.

However, moment methods require a good understanding of the moments of the random objects under consideration. While moment 
estimates for random permutations boil down to combinatorial formulas, the situation is more involved for random unitary or 
orthogonal matrices. In these cases, formulas for moments require Weingarten calculus, as developed in
\cite{MR2217291,MR1959915}.
One of our key results relies on the comparison of moments of unitary random matrices and Gaussian random matrices, which 
extends the line of research of \cite{MR2257653,MR3769809} and others. 
Our main result in this direction is Theorem \ref{theorem-with-brackets}, which quantifies how well unitary random matrices
can be estimated by Gaussian matrices. 
A critical preliminary result to achieve this goal is to develop a systematic method to handle centering, for which we use a bracket notation. 
This is based on another notion of lattices, and a similar moment theoretical work was initiated in  \cite{MR2830615}.
Our main result here is Theorem \ref{lemma:unitary-Wg-expansion3}.

The paper is organized as follows. This introduction is followed by Section \ref{sec:main}, which states the main results and describes the steps
of the proof and its applications. 
 Section \ref{sec:Wg} contains the necessary new developments on uniform centered Weingarten functions.
 Section \ref{sec:SAFNB} describes the non-commutative non-backtracking theory in full generality.
 Section \ref{sec:trace} is devoted to trace estimates based on the previous sections and the completion of the proof of the main theorems.
We finish the paper with Section \ref{sec:appendix}, which is an appendix in which we provide a constructive proof of Pisier's unitary linearization trick, on which this paper relies heavily. 

\paragraph{Acknowledgements} 
This paper was initiated while the first author was a visiting JSPS scholar at KU in 2018, and he acknowledges the hospitality of 
JSPS and of Kyoto University. 
BC was supported by JSPS KAKENHI 17K18734 and 17H04823 and CB by ANR grant ANR-16-CE40-0024.
We are indebted to Narutaka Ozawa for drawing our attention to reference \cite{MR2553116}. 
We also are very grateful to the anonymous referees for their efforts in improving the manuscript.

\section{Strong asymptotic freeness for random unitary tensors}
\label{sec:main}

\subsection{Main result}

In the sequel, for an integer $k \geq 1$, we set 
$\INT{k} = \{1, \ldots, k \}.$ 
For any compact group, there exists a unique probability measure that is left- and right-invariant under translation; it is called the 
normalized Haar measure.  In this paper, we are interested in the case of the unitary group $\Un$ and the orthogonal group $\On$. 

We focus on the unitary case to keep the paper to a reasonable length. The orthogonal case boils down to technical modifications of the unitary case, which
we summarize in Subsection \ref{subsec:orthogonal-variations}.
Let $d \geq 2$ be an integer and let $U_1, \cdots,U_{d}$ be independent Haar-distributed random 
unitary matrices in $\Un$.  For $d+1 \leq i \leq 2d$, we set $U_i = U_{i-d}^*$. We consider the involution on $\INT{2d}$, defined by 
$i^* = i+d$ for $1 \leq i \leq d$ and $i^* = i-d$  for $d+1 \leq i \leq 2d$. We thus have $U_{i^*}  = U_i^* $ for all $i \in \INT{2d}$.

If $M \in M_n(C)$ and $\veps \in \{\cdot, -\}$, we set $M^{\veps} = M$ for $\veps = \cdot$ and $M^\veps = \bar M$ for $\veps = -$. 
We fix a triple of integers $q = q_- + q_+ \geq 1$. For each $ i \in \INT{2d}$,  we introduce the unitary matrix in $\Unq$,
\begin{equation}\label{eq:defVi}
V_i = \bar U_i^{\otimes q_-} \otimes U_i^{\otimes q_+}  = U_i^{\veps_1} \otimes \cdots \otimes U_i^{\veps_{q}},
\end{equation}
where $(\veps_1,\ldots, \veps_{q } )\in \{.,-\}^ q$ is the symbolic sequence: $\veps_p = -$ for $p \leq q_-$ and $\veps_p  = \cdot$ otherwise.  
We recall the following fact:
\begin{proposition}
The average matrix $\dE V $
is
the orthogonal projection, 
denoted by $P_{H}$,
onto the vector subspace $H$ of elements invariant under (left) multiplication by  
 $\bar U ^ {q_-} \otimes U ^ {q_+}$ for all $U \in \Un$.
\end{proposition}
Although this result is classical and its proof is elementary, it is not very well referenced in our specific context, so we include an
outline for the sake of being self-contained, see also \cite[Section III.A]{MR2553116}. 

\begin{proof}[Outline of the proof]
The fact that $\dE V $ is a projection 
follows from a direct application of the fact that the distribution of a Haar-distributed element
$U$ is the same as 
$\tilde U\cdot \hat U$ where $\tilde U$ and $\hat U$ are two iid copies of $U$.
The fact that it is self-adjoint follows from the fact that $U$ and $U^*$ have the same distribution. 
The fact that every invariant vector is invariant under the mean is trivial. 
Finally, let $x\notin H$. By definition, there exists $V_0$ such that $V_0x\ne x$. Since $V$ is a unitary operator and therefore preserves the Euclidean norm,
it follows from the strict convexity of the Euclidean norm that $\|\dE V x\|_2 <\|x\|_2$ and therefore $x$ is not in the image of $\dE V$, which 
concludes the proof.
\end{proof}
We refer to
Subsection \ref{subsec:Wg} for a method to compute $\dE V$. For example if the vector $(\veps_1,\ldots, \veps_{q } )$ is not balanced, 
that is $q_- \ne q_+$ then $\dE V = 0$. We set 
\begin{equation}\label{eq:bracket}
[V_i] = V_i - \dE V_i = V_i- P_{H}.
\end{equation}

For a fixed integer $r \geq 1$, let $(a_0, a_1, \ldots, a_{2d})$ be matrices in $M_r(\dC)$.  We introduce the following matrix on 
$\dC^r \otimes \dC^{n^q}$: 
\begin{equation}\label{eq:defA}
A = a_0 \otimes 1 + \sum_{i=1}^{2d} a_i \otimes V_i,
\end{equation}
where $1$ is the identity. From what precedes, the vector space $H_r =  \dC^r \otimes H$ is an invariant subspace of $A$ and its adjoint $A^*$. 
We denote by $H_r ^\perp$ the orthogonal of $H_r$ and denote by $A_{|H_r^\perp}$ the restriction of $A$ to this vector space. Our goal 
is to describe the spectrum of $A_{|H_r^\perp}$ as $n$ goes to infinity.

To this end, the key observation is the following. Consider the unitary representation of the free group $\Fd $ on $ \dC^{n^q}$ defined 
by $\pi(g_i) = V_i$, where $(g_i)_{i \in \INT{2d}}$ are the free generators and their inverses: $g_{i^*} = g_i ^{-1}$. Note that 
this representation $\pi$ is random and depends implicitly on $n$ and $(q_-,q_+)$.   The matrix $A$ is the image by $\pi$ of the 
following operator in $\ell^2(\Fd)$ in the (left)-group algebra on $\Fd$:
\begin{equation}\label{eq:defAfree}
A_\star = a_0 \otimes 1 + \sum_{i=1}^{2d} a_i \otimes \lambda(g_i),
\end{equation}
where $g \mapsto \lambda (g)$ is the left-regular representation (that is, left multiplication by group elements).

The main technical result of this paper is the following theorem. In the sequel $\| \cdot \|$ denotes the operator norm of an operator in 
a Hilbert space: $\| T \| = \sup_{x \ne 0} \| Tx \|_2 / \|x \|_2$ where $\| x\|_2$ is the Euclidean norm in the Hilbert space.  

\begin{theorem}\label{th:main}
There exists a universal constant $c >0$ such that the following holds. Let $r \geq 1$ be an integer and let  $(a_0, a_1, \ldots, a_{2d})$ 
be matrices in $M_r(\dC)$. If $q = q(n)$ is a sequence such that $q \leq c \ln(n) / \ln (\ln (n))$ for all $n$ large enough, then with 
probability one, we have 
$$
\lim_{n \to \infty} \| A_{|H_r^\perp} \| = \| A_\star \|.
$$
\end{theorem}
Thanks to the  ``linearization trick" of the theory of operator algebras (see Pisier \cite{MR1401692} and the monograph by 
Mingo and Speicher \cite[p256]{MR3585560} for an accessible treatment), Theorem \ref{th:main} has 
an apparently much more general corollary that we describe in the sequel of this subsection.
Although Pisier's linearization trick is enough for our paper, we include, for the convenience of the reader, an appendix in Section \ref{sec:appendix}, which provides a self-contained constructive proof of the linearization theorem.
A reader interested in linearization tricks might find this section of independent interest.
In the simplest case where $r=1$, all $a_i$'s are equal to one and $q_- = q_+$, then Theorem \ref{th:main} is contained 
in \cite[Theorem 6]{MR2553116}. In this case, a more direct method of moments based on Schwinger-Dyson equations can be performed.

Let us start with some basic notation of representation theory.
For any $n$, let $\rho = \rho (n)$ be a signature (to lighten the notation, we omit the dependence in $n$). 
More precisely, $\rho$ is a pair of  Young diagrams: two pairs of non-negative integer sequences $\lambda =(\lambda_1,\lambda_2,\ldots)$ and 
$\mu= (\mu_1,\mu_2,\ldots)$ satisfying the following properties: 
$\mu_i \geq \mu_{i+1}$, $\lambda_i\geq \lambda_{i+1}$ and $\sum_i\lambda_i+\mu_i<\infty$.
We introduce $l(\lambda)=\max\{i, \lambda_i>0\}$
and $|\lambda |=\sum_i\lambda_i$ (and likewise for $\mu$). For notation, we refer, for example to \cite{MR0473098}.
If $n\ge l(\lambda)+l(\mu)$ we call
$V_{n,\rho}$ the Hilbert space such that the group homomorphism $\tilde \rho: \Un \to \mathbb{U}(V_{n,\rho})$ 
is the
 irreducible representation of $\Un$ whose highest weights are 
 $$\lambda_1\ge\cdots\ge\lambda_{l(\lambda )}\ge 0\ge \cdots \ge 0
\ge -\mu_{l(\mu)}\ge\ldots \ge -\mu_{1}.$$ 
The signature formulation is a reformulation of the highest weight theory, but it is convenient as it allows
to define simultaneously representations for all unitary groups $\Un$ with $n\ge l(\lambda)+l(\mu)$.
We call $\mathbb{U}_{n,\rho}=\tilde\rho (\Un)$. It is the quotient of 
$\Un$ under the above representation map. 
We view it as a matrix subgroup of $\mathbb{U}(V_{n,\rho})$.
Letting  $U_i$ be a Haar-distributed random matrix in $\Un$, we call $W_i:=\tilde\rho (U_i)$ its image under the irreducible representation map associated with $\rho$.
Since the map is surjective, $W_i$ is Haar-distributed according to $\mathbb{U}_{n,\rho}$.
Consider any operator on $\dC^r \otimes \ell^2 (\Fd)$ of the form 
$$
P_\star = \sum_{g \in \Fd} a_g \otimes \lambda (g),
$$
where for all $g \in \Fd$,  $a_g$ is a matrix in $M_r(\dC)$. We assume that $a_g$ is non-zero for a finite number of group elements. 
In other words, $P_\star$ is a matrix-valued non-commutative polynomial. 
The image of $P_\star$ by the representation $\pi$ is denoted 
by $P$:
$$
P = \sum_{g \in \Fd} a_g \otimes W(g),
$$
where $W(g) = W_{i_1} \ldots W_{i_k}$ if $g = g_{i_1} \ldots g_{i_k}$. Obviously, $H_r$ and $H_r^\perp$ are again invariant subspaces 
of $P$ and $P^*$. The operator $P_\star$ (and thus $P$) is self-adjoint if the following condition is met: 
\begin{equation}
\label{eq:symP}
a_{g^{-1}} = a_g^* \quad \hbox{ for all } g \in \Fd.
\end{equation}
The following result is a corollary of Theorem \ref{th:main}. 
\begin{corollary}\label{cor:main}
Let $r \geq 1$ be an integer, $P$ and $P_\star$ be as above, and $c$ be as in Theorem \ref{th:main}. If $q = q(n)=|\lambda |+|\mu |$ 
is a sequence such that 
$q \leq c \ln (n) / \ln (\ln (n))$ for all $n$ large enough, then with probability one, we have 
$$
\lim_{n \to \infty} \| P_{|H_r^\perp} \| = \| P_\star \|.
$$
Moreover, if \eqref{eq:symP} holds then, with probability one, the Hausdorff distance between the spectrum of $P_{|H_r^\perp}$ and 
$P_\star$ goes to $0$ as $n$ goes to infinity. 
\end{corollary}

The reason why
Theorem \ref{th:main} is sufficient to allow general polynomials in the above Corollary \ref{cor:main}
follows from \cite[Section 6]{MR4024563}. 
This is a consequence of \cite[Proposition 6]{MR1401692}.
Details can also be found in the appendix of section \ref{sec:appendix}.
As for the fact that we can replace tensors with irreducible representation, this follows from the following two facts: 
(i) contracting by an orthogonal projection does not increase the operator norm, and (ii) there exists an orthogonal projection $Q$ 
that commutes with all $V_i$'s such that on $\Im(Q)$, $W_i=QV_iQ$, when $q_+=|\lambda |, q_-=|\mu|$.
We refer, for example, to \cite{MR2522486} for details on the tools involved, such as Schur-Weyl duality.

Corollary \ref{cor:main} establishes the almost-sure strong asymptotic freeness of unitary representations of independent 
Haar-distributed unitary matrices provided that the unitary representation is a sub-representation of a tensor product representation of 
not too large dimension. We observe that this last point is critical. Indeed, the determinant is a representation of dimension $1$ (which is 
a sub-representation of the tensor product representation with $q = n$), and there is no freeness in nontrivial finite-dimensional spaces.

\subsection{Overview of the proof}

We now explain the strategy behind the proof of Theorem \ref{th:main}. The strategy is the same as the one used in  \cite{MR4024563} 
for permutation matrices. The model in  \cite{MR4024563} is, however, very different, and new technical achievements were necessary 
for tensor products of unitary matrices.

First of all, from \cite{MR4024563}, it is enough to prove that, for any $\veps > 0$, we have 
\begin{equation}\label{eq:tbdA}
\| A_{|H_r^\perp} \| \leq  \| A_\star \| + \veps,
\end{equation}
with high probability. To achieve this, it is possible to restrict ourselves to self-adjoint operators. Indeed, the operator norm of an 
operator $M$ is also the square root of the right-most point in the spectrum of the non-negative operator
$$
\begin{pmatrix}
0 & M \\
M^*& 0 
\end{pmatrix} = E_{12} \otimes M + E_{21} \otimes M^*,
$$
where $E_{ij}$ is the $2 \times 2$ canonical matrix whose entry $(i,j)$ is equal to $1$ all other entries are equal to $0$. In particular, 
if $A$ is of the form \eqref{eq:defA}, at the cost of changing $r$ in $2r$, we may assume without loss of generality that 
\begin{equation}
\label{eq:symA}
a_{i^*} = a_{i}^* \quad \hbox{for all $i \in \INT{2d}$}.
\end{equation}

For such a given operator $A$, we will define a family of companion operators, denoted by $B_\mu$ and indexed by a real parameter 
$\mu$, with the property that if, for all $\veps >0$, with high probability, for all $\mu$, we have
\begin{equation}\label{eq:tbdB}
\rho( B_\mu) \leq  \rho((B_\star)_\mu) + \veps
\end{equation}
then \eqref{eq:tbdA} holds. In the above expression, $(B_\star)_\mu$ is the companion operator of $A_\star$ defined in 
\eqref{eq:defAfree} and $\rho(M) = \sup \{ |\lambda| : \lambda \in \sigma(M)\}$ is the spectral radius. These operators $B_\mu$ are 
called the non-backtracking operators. This will be explained in Section \ref{sec:SAFNB}. This part is an extension of 
\cite{MR4024563} in a more general setting.  

It is not a priori obvious why the claim \eqref{eq:tbdB} is easier to prove than claim \eqref{eq:tbdA}. A reason is that the powers 
of $(B_\star)_\mu$ are much simpler to compute than the powers of $A_\star$. 

In Section \ref{sec:trace}, we prove that \eqref{eq:tbdB} holds by using the expected high trace method popularized by F\"uredi and 
Koml\'os \cite{MR637828} in random matrix theory. It can be summarized as follows: assume that we aim at an upper bound of the 
form $\rho(M) \leq ( 1 + o (1)) \theta$ for some $\theta >0$ and $M \in M_n (\dC)$ random. We observe that for any $\ell$ integer, 
$$
\rho ( M)^{2\ell} \leq \| M^{\ell}  \|^2 =  \| M^\ell (M^*)^\ell \| \leq \tr (  M^\ell (M^*)^\ell).
$$
Moreover, at the last step, we lose a factor at most $n$ (and typically of this order). If we can prove that  
$$
\EE \tr (  M^\ell (M^*)^\ell) \leq n \theta^{2\ell},  
$$
then we will deduce from Markov's inequality that the probability that $\rho(M) \leq n^{1/(2\ell)} ( 1+ \delta) \theta$ is at least 
$1 - (1+ \delta)^{-\ell}$. In particular, if $n^{1/(2\ell)} \to 1$, that is $\ell \gg \ln(n)$, then this last upper bound is sharp enough 
for our purposes.

 A usual strategy to evaluate $ \EE \tr (  M^\ell (M^*)^\ell)$ is to expand the trace and the powers as the sum of product matrix entries 
 and then use the linearity of the expectation. We thus need to combine two ingredients: (i) a sharp upper bound on the expectation of 
 the product of matrix entries in terms of combinatorial properties of the entries and (ii) a counting machinery to estimate the number 
 of entries in the trace that have the given combinatorial properties. 

In our setting, the matrix $M$ is the non-backtracking matrix $B_\mu$ and $\theta = \rho( (B_\star)_\mu)$ for a fixed\,$\mu$. Among 
the difficulties, the value of $\theta$ is unknown exactly, and also, a probabilistic control of an event where all $\mu$ are considered is 
needed in \eqref{eq:tbdB}. These issues were already present in \cite{MR4024563}.  Here, the presence of the tensor 
products will also create a significant complication in the counting arguments of ingredient (ii). An important step for ingredient (i) is performed 
in Section \ref{sec:Wg}, where we prove a new estimate on the expectation of the product of a large number of entries of a Haar unitary 
matrix. This will be done by developing recent results on Weingarten calculus. 
   
\section{High order Gaussian approximation for random unitary\;matrices} 
\label{sec:Wg}

This section aims to develop an efficient machinery to compare the expectation of products of entries of a Haar-distributed unitary 
matrix with the average of the same product when we replace the unitary matrix with a complex Gaussian matrix. The main results of 
these sections are Theorem \ref{theorem-warmup} and Theorem \ref{theorem-with-brackets} below. 
   
\subsection{Wick calculus}

In this subsection, we recall the classical Wick formula. We then introduce a centered version which is new.

\subsubsection{Wick formula}

Let $V$ be a real vector space of real centered Gaussian variables.
It is called a real Gaussian space. 
Similarly, let $W$ be a complex vector space of 
complex centered Gaussian variables. It is called a complex Gaussian space. 
 In both cases, the addition is the regular addition of real (respectively complex) valued random variables, and likewise for the scaling.  

Let $(x_i)_{i \in I}$ be iid real centered Gaussian variables indexed by a countable set $I$. Then $\SPAN ((x_i)_{i \in I})$ is a real vector space, 
and any real vector space can be realized in this way. 
Finally, a real Gaussian vector space comes with the scalar product $(x,y)\mapsto \EE(xy)$ and 
similarly, a complex Gaussian vector space comes with the Hilbert product  $(x,y)\mapsto \EE(\overline{x}y)$.

Wick's Theorem asserts that the scalar product is enough to recover the structure of the Gaussian space completely, 
in other words, there is a one-to-one correspondence (in law) between Gaussian spaces and their Hilbert structure; see Janson \cite[Chapter 3]{MR1474726}.
Everything relies on a moment formula (which, in the case of Gaussian variables, determines the distribution).
In the real case, 
$$\EE(x_1\ldots x_k)=\sum_{p\in P_k} \EE_p(x_1,\ldots , x_k),$$
where $P_k$ is the collection of pair partitions of $\INT{k}$, typically denoted by
\begin{equation}\label{eq:notep}
p= \{ p_1 , \ldots, p_{k/2} \} = \{\{i_1,j_1\},\ldots , \{i_{k/2},j_{k/2}\}\}
\end{equation}
with $i_l<j_l$ and $i_{l}<i_{l+1}$ (obviously $P_k$ is empty when
$k$ is odd), and, under this notation, 
$$\EE_p(x_1,\ldots , x_k)=\EE(x_{i_1}x_{j_1})\ldots \EE(x_{i_{k/2}}x_{j_{k/2}}).$$
In the complex case,
\begin{equation}\label{eq:WickC}
\EE(g_1\ldots g_k \overline h_1\ldots \overline h_k)=\sum_{\sigma\in S_k} \EE_{\sigma} ( g_1,\ldots , g_k , h_1, \ldots  , h_k).
\end{equation}
where $S_k$ is the permutation group on $\INT{k}$ and  
$$
\EE_{\sigma} ( g_1,\ldots , g_k , h_1,\ldots  , h_k) =  \prod_{l=1}^k  \EE(g_l \overline h_{\sigma(l)}).
$$
Note that the complex case  can be deduced directly from the real case thanks to the real structure of complex Gaussian spaces 
($S_k$ appears as a subset of $P_{2k}$ by identifying a permutation $\sigma \in S_{2k}$ with the partition $p = \{p_1, \ldots, p_{k} \}$ 
of $P_{2k}$ with $p_l = \{ l , k + \sigma(l)  \}$).
 As for the real case, the formula is trivial when the $x_i$'s are all the same (this is the formula for the moments of a Gaussian).
For the general case, after observing that both the right-hand side and the left-hand side are $k$-linear and symmetric, we conclude
by a polarization identity. 

\subsubsection{Centered Wick formula}

If $X$ is a vector valued integrable random variable, we call $[X]$ the centered random variable 
\begin{equation}\label{eq:defbracket}
[X] = X-\EE(X).
\end{equation}

For integer $T \geq 1$, the Wick formula can  be extended as follows: if $\pi = (\pi_t)_{t\in \INT{T}}$ is a partition of $\INT{k}$, 
we have, for real Gaussian variables,
$$
\EE \PAR{ \prod_{t=1}^T \big[ \prod_{i \in \pi_t} x_{i} \big] } = \sum_{p \in P_{ k}(\pi)}  \EE_p(x_1,\ldots , x_k),
$$
where $  P_{k}(\pi)$ is the subset of $P_k$  of pair partitions 
$p=\{p_1,\ldots ,p_{k/2}\}$ with the following property: for each block $\pi_t $ of the partition, 
there exists at least one $j$ such that the pair $p_j$ of $p$  has one element in $\pi_t$
and one element outside. 
Although this formula is not mainstream in probability theory, it follows directly from 
standard inclusion-exclusion-type formulas (see Lemma \ref{le:gaussiansieve} below for the complex case).
A similar formula holds for a complex Gaussian space: 
$$
\EE \PAR{ \prod_{t=1}^T \big[ \prod_{i \in \pi_t} g_{i} \bar h_i \big]}=\sum_{\sigma\in S_k(\pi)}   \EE_{\sigma} ( g_1,\ldots , g_k , h_1, \ldots  , h_k),$$
where $S_k(\pi)$ is the subset of $S_k$ of permutations $\sigma$ such that for each block $\pi_t$, there exists at least one $i \in \pi_t$ 
such that $\sigma(i) \notin \pi_t$.  
We will use these results in the proof of Theorem \ref{theorem-with-brackets}
when we compare Gaussian moments and unitary moments.

\subsection{Weingarten calculus with brackets}
\label{subsec:Wg}
\subsubsection{Unitary Weingarten formula}

In the sequel, if $x= (x_1, \ldots, x_k)$ and $y = (y_1, \ldots, y_k)$ are multi-indices in $\INT{n}^k$, for $\sigma\in S_k$ we set 
$$\delta_{\sigma}(x,y) = \prod_{l=1}^k \delta_{x_l,y_{\sigma (l)}},$$ 
where $\delta_{i,j}$ is the usual Kronecker delta ($1$ if $i=j$ and $0$ otherwise).
Similarly, if $k$ is even and $p \in P_k$ is a pair partition, 
$\delta_{p}(x)$ is a generalized Kronecker delta: namely, $\delta_{p}(x)$ takes the value $1$ if the map $i\mapsto x_i$ is constant on each block of $p$ and zero in all other cases.  This is a product of $k/2$ 
such deltas. 
 
We start by recalling the analogue of Wick calculus for the normalized Haar measure on the unitary group $\Un$.
We refer, for example, to \cite{MR1959915} for early versions of this result.

 \begin{theorem}\label{th:unitary-Wg}
Let $k,n \geq 1$ be integers and $x,y,x',y'$ in $\INT{n}^k$. There exists a function $\WgU (\cdot, \cdot ,  n ) : S_k\times S_k\to \mathbb{R}$ such that if $U = (U_{ij})$ is Haar-distributed on $\Un$,
$$\EE \PAR{ \prod_{l=1}^k U_{x_l y_l}\overline U_{x_l'y_l'}} =
\sum_{p,q\in S_k}\delta_p(x,x')\delta_q(y,y')\Wg(p,q,n).$$
This function is uniquely defined iff $k\le n$; see Theorem
\ref{lemma:unitary-Wg-expansion1} for a formula.
\end{theorem}

The sequel of this subsection is devoted to providing a modern description of $\WgU$.
It follows from commutativity that 
$\WgU(p,q,n)$ actually only depends on the conjugacy class of $pq^{-1}$ so in the unitary case we will 
write $\WgU(pq^{-1},n)=\WgU(p,q,n)$, so
$\WgU$ becomes a central function on $S_k$. 

Let $\sigma\in S_k$. We denote by $|\sigma| = k - \ell(\sigma)$ where $\ell(\sigma)$ is the number of disjoint cycles in the cycle 
decomposition of $\sigma$. Classically, $|\sigma|$ is also the minimal number $m$ such that $\sigma$ can be written as a product 
of $m$ transpositions. In particular $(-1)^{|\sigma|}$ is the signature of $\sigma$. 
For $l \geq 0$, we call $P(\sigma,l)$ the collection of solutions of 
\begin{equation}\label{hurwitz}
\sigma= (i_1,j_1)\cdots (i_{|\sigma|+l}, j_{|\sigma|+l}),
\end{equation}
where $j_p\le j_{p+1}$ and $i_p<j_p$. Note that $P(\sigma,l) = 0$ unless $l = 2g$ is even (since composing by a single 
transposition negates the signature).

The following theorem is a combination of Theorem 2.7 and Lemma 2.8 in \cite{MR3680193}  (beware that the definition of the map 
$l \to P(\sigma,l)$ is shifted by $|\sigma|$ in this reference), see also \cite{MR3010693}.  In the statement below and in the sequel, if $S$ 
is a finite set, we denote by $|S|$ its cardinal number (not to be confused with $|\sigma|$ for $\sigma \in S_k$). 
\begin{theorem} \label{lemma:unitary-Wg-expansion1}
For $\sigma \in S_k$, we have the  
expansion
\begin{equation} \label{eq:unitary-Wg-expansion1}
\WgU(\sigma ,n)= (-1)^{|\sigma|} n^{-k - |\sigma|}  \sum_{g \ge 0} | P(\sigma, 2g) | n^{-2g}.
\end{equation}
This expansion is formal in the sense that $\Wg(\sigma ,n)$ is a rational fraction in $n$, and its
power series expansion in the neighborhood of infinity is as above. 
\end{theorem}

Since the poles of $n\to \WgU (\sigma,n)$ are known to be in the set $\{-k+1,\ldots , k-1\}$, it follows
that the power series expansion is convergent as soon as $n\ge k$.
The following result can be found in \cite[Theorem 3.1]{MR3680193}; it is an estimate of the number of solutions of \eqref{hurwitz}. 
It allows subsequently to give estimates on the Weingarten function. 

\begin{proposition}\label{combinatorial-estimate}
Let $k$ be a positive integer. 
For any permutation $\sigma \in S_k$ and 
integer $g \geq 0$, we have
$$
 (k-1)^g | P (\sigma , 0) | \leq 
|P(\sigma ,  2g)|\leq 
\PAR{ 6 k^{7/2}}^g |P (\sigma , 0)|.
$$
\end{proposition}

For our purposes, we will need the following corollary: 

\begin{corollary}\label{cor-upper-bound-unitary}
If $\sigma \in S_k$ and $12 k^{7/2}  \leq n^{2}$, 
$$
| \WgU(\sigma ,n) | \leq \PAR{ 1 + 24k^{7/2} n^{-2} } n^{-k - |\sigma|} 4^{|\sigma|}. 
$$
\end{corollary}

\begin{proof}
By \cite[Corollary 2.11]{MR3010693}, we have $$| P(\sigma , 0)| = \prod_{i=1}^{\ell (\sigma)} C_{\mu_i - 1}, $$
 where $C_n$ is the $n$-th Catalan number and $\mu_i$ is the length of the $i$-th cycle of $\sigma$. Since $C_n \leq 4^n$, we get  
 $| P(\sigma , 0)| \leq 4^{|\sigma|}$. From  Proposition \ref{combinatorial-estimate}, we deduce that  for any integer $g \geq 0$,
\begin{equation}\label{eq:upper-bound-unitary}
 | P(\sigma , 2g)| \leq  4^{|\sigma |}\PAR{ 6k^{7/2}}^{g}.
 \end{equation}
From Theorem \ref{lemma:unitary-Wg-expansion1}, we deduce that 
$$
| \WgU(\sigma ,n) |  \leq  n^{-k - |\sigma|} 4^{|\sigma|} \sum_{g\geq 0} \PAR{ \frac{6 k^{7/2}}{n^2} }^g .
$$
The conclusion follows by using that $(1-x)^{-1} \leq 1 + 4 x$ for all $0 \leq x \leq 1/2$.
\end{proof}

\subsubsection{The centered case}

For a symbol $\veps \in \{ \cdot,-\}$ and $z \in \dC$, we take the notation that $z^{\varepsilon}=z$ if $\varepsilon=\cdot$ and 
$z^{\varepsilon}=\overline z$ if
$\varepsilon = -$.
Our goal is to compute, for $U = (U_{ij})$ Haar-distributed on $\Un$, expressions of the form
\begin{equation}\label{eq:prodbraex}
\EE \PAR{\prod_{t=1}^T \big[  \prod_{l = 1}^{k_t} U_{x_{tl} y_{tl }}^{\varepsilon_{tl}} \big]},
\end{equation}
where we have used the bracket defined in \eqref{eq:defbracket} --
and then to estimate it in a useful way. 

The polynomial to be integrated can be expanded into $2^T$ terms for which the Weingarten formula can be applied 
each time. However, such an approach does not yield good estimates because the sum is signed, 
and additional cancellations occur, which results in the 
items to be summed not having the correct decay in the large dimension. On the other hand, 
for a given pairing of indices, one can group the $2^T$ (signed) Weingarten functions into one single more general Weingarten function whose expansion turns out to be non-signed and, therefore, much more suitable for 
more sophisticated asymptotics.
The following proposition addresses this issue. 

\begin{proposition}\label{weingarten-centered}
Let $k$ be even, $\pi = (\pi_t)_{t \in \INT{T}}$ be a partition of $\INT{k}$, let $\veps \in \{\cdot,-\}^k$ be a balanced sequence (in the sense that it has as many $\cdot$ than $-$) and $x, y$ in $\INT{n}^k$.  
There exists a generalized Weingarten function
$\WgU[\pi] (p,q,n)$ such that
\begin{equation*}\label{eq:Wgcentered}
\EE \PAR{\prod_{t=1}^T \big[  \prod_{i \in \pi_t} U_{x_{i} y_{i}}^{\varepsilon_{i}} \big]}
=\sum_{p,q\in P^\veps_k}\delta_{p}(x)\delta_{q}(y)\WgU[\pi] (p,q,n),
\end{equation*}
where $P_k^\veps \subset P_k$ are the pair partitions that match an $\veps_i =\cdot$ with an $\veps_{i'} = -$  
(seen as bijections from the set of $i$'s such that $\veps_i = \cdot$ to the set of $i$'s such that $\veps_i = -$). 
\end{proposition}

\begin{proof}
Let $A\subset \INT{T}$ and $X_1,\ldots ,X_T$ be random variables. 
We introduce the following temporary notation:
$$\EE_A(X_1,\ldots , X_T)=\EE\PAR{\prod_{t \in A}X_t}\prod_{t \notin A}\EE(X_t).$$
It follows from this definition that
\begin{equation}\label{eq:prodbracket}
\EE([X_1]\ldots [X_T])=\sum_{A\subset \INT{T}}\EE_A(X_1,\ldots , X_T)(-1)^{T - |A|}.
\end{equation}
We apply this equation to
$$X_t:=\prod_{i \in \pi_t} U_{x_{i} y_{i}}^{\varepsilon_{i}}.$$
We claim that for a given $A$, there exists a function $\WgU_A[\pi] (p,q,n)$ such that
$$\EE_A ( X_1,  \ldots, X_T )
=\sum_{p,q\in P_k^\veps}\delta_{p}(x)\delta_{q}(y)\WgU_A[\pi] (p,q,n).$$

To define $\WgU_{A}$ precisely, we call $\pi_A$ the partition of 
$\INT{k}$ whose blocks are  $\pi_{t},  t\notin A$,
and a last block that complements the previous ones (in other words, we merge all blocks in $A$ and leave the other unchanged). 
From Theorem \ref{th:unitary-Wg}, this yields the following description of $\WgU_{A}$:
if either the permutation $p$ or the permutation $q$ fails to respect the partition 
$\pi_A$ (that is, does not leave all blocks invariant), then its value
is zero. Or else, it is a product of Weingarten functions obtained by restricting the permutations over the block of $\pi_A$: 
\begin{equation}\label{eq:WgA}
\WgU_A[\pi] (p,q,n) = \prod_{b \in \pi_A} \WgU (p_b,q_b,n),
\end{equation} where the product is over all blocks $b$ of $\pi_A$ and $p_b,q_b$ are the restrictions of $p,q$ to the block $b$.
The fact that the above holds follows from an application of the Weingarten formula, Theorem \ref{th:unitary-Wg}, for
each expectation factor appearing in the product. 

In turn, the explicit formula we obtain for the generalized Weingarten formula becomes
\begin{equation}
\label{eq:WgU}\WgU[\pi] (p,q,n)= \sum_{ A\subset \INT{T}}(-1)^{T - |A|}\WgU_{A}[\pi] (p,q,n),
\end{equation}
and this concludes the proof.
\end{proof}

\begin{example}\normalfont 
We take $T = 2$ and for $k_1 \in \INT{k}$, $\pi_1 = \INT{k_1}$, $\pi_2 = \INT{k} \backslash \INT{k_1}$. Then, if $p$ or $q$ do not 
leave invariant $\pi$, we have $\WgU[\pi] (p,q ,n)=\WgU(p,q,n)$. Otherwise, we get 
$\WgU[\pi] (p,q,n)=\WgU(p,q,n)-\WgU(p_{\pi_1},q_{\pi_1},n)\WgU(p_{\pi_2},q_{\pi_2},n)$, where $p_{\pi_t}$ is the restriction of 
$p$ to the block $\pi_t$. \end{example}

We now give an analogue of Theorem \ref{lemma:unitary-Wg-expansion1}. We consider the setting of Proposition \ref{weingarten-centered}. 
Let $\veps \in \{\cdot,-\}^k$ be a balanced sequence and  $\pi$ be a partition of $\INT{k}$. Recall that we identify pair partitions 
$p,q \in P_k^\veps$ with bijections from the $i$'s such that $\veps_i = \cdot$ to the set of $i$'s such that $\veps_i = -$. For 
$p,q \in P^\veps_k$, we call 
$P[\pi] (p,q, l)$ the collection of solutions of
\begin{equation}\label{hurwitz-centered}
p= (i_1,j_1)\ldots (i_{|p q^{-1}|+l}, j_{|p q^{-1}| + l})q
\end{equation}
where $j_r\le j_{r+1}$, $i_r<j_r$ 
and that satisfies the following property: 
the solution can be restricted to \emph{no single block} of $ \pi$ in the  
sense 
that
if there exists a block $b$  of $  \pi$ such that $p,q$ and each transposition $(i_r,j_r)$ leave $b$ invariant, 
then the solution is not in $P[\pi] (p ,q , l)$, otherwise, the solution is in $P[\pi] (p ,q , l)$.

These notations allow us to reformulate combinatorially the centered Weingarten function used in
Proposition \ref{weingarten-centered}:

\begin{theorem} \label{lemma:unitary-Wg-expansion3} Let $k$ be  even, $\pi = (\pi_t)_{t \in \INT{T}}$ be a partition of $\INT{k}$, let $\veps \in \{\cdot,-\}^k$ be a balanced sequence. For all $p, q \in P^\veps_k$, we have the   expansion
\begin{equation} \label{eq:unitary-Wg-expansion3}
\WgU [\pi](p,q ,n)=(-1)^{|pq^{-1} |} n^{-k/2-|pq^{-1} | }  \sum_{g \ge 0}  | P[\pi](p,q, 2g)| n^{-2g}. 
\end{equation}
\end{theorem}

\begin{proof}
We set $\sigma = p q^{-1}$, it is a bijection on the set $\dot{\INT{k}}$ of elements $\INT{k}$ such that $\veps_i = \cdot$. 
For $A \subset \INT{T}$,  we start with the formula for $\Wg_A[\pi]$ in \eqref{eq:WgA}, where we recall that $\pi_A$ is the partition 
obtained for $\pi = (\pi_t)$ by merging all blocks in $A$. We apply Theorem \ref{lemma:unitary-Wg-expansion1} to each term  
$\WgU (p_b,q_b,n) = \WgU (p_b q^{-1}_b,n)$ in the product \eqref{eq:WgA}. One finds that
\begin{equation}\label{eq:WgAex}
\Wg_A[\pi](p,q,n) = (-1)^{|\sigma|} n^{-k/2 - |\sigma|}   \sum_{g \ge 0}  | P_A[\pi](p,q, 2g)| n^{-2g},
\end{equation}
where $P_{A}[\pi](p,q, 2g)$ is  the empty set if either $p$ or $q$ do not respect the
partition $\pi_A$ (that is, do not leave all blocks of $\pi_A$ invariant)
and else, it is the collection of solutions of \eqref{hurwitz-centered}, 
$\sigma= (i_1,j_1)\ldots (i_{|\sigma| + 2g}, j_{|\sigma| + 2g})$
where $j_p\le j_{p+1}$ and $i_p<j_p$ in $\dot{\INT{k}}$ and every transposition $(i_p,j_p)$ respects the partition $\pi_A$.
This follows from the facts that the condition ``$j_p\le j_{p+1}$ and $i_p<j_p$'' is a total order on transpositions and that
all other quantities in the sum \eqref{eq:unitary-Wg-expansion1} are multiplicative over blocks of $\pi_A$.

We observe from the definitions that $P_{\INT{T}} [\pi] ( p, q, 2g) = P (pq^{-1},2g) $ and that
$$
P[\pi](p,q,2g) = P (pq^{-1},2g) \backslash \bigcup_{ t = 1 }^T P_{\INT{T}\backslash \{t\} } [\pi] (p,q,2g).
$$
Indeed, $P_{\INT{T}\backslash \{t\} } [\pi] (p,q,2g)$ is the set of solutions of \eqref{hurwitz-centered} such that $p$, $q$ and the 
transpositions leave $\pi_t$ invariant. In addition, from the definition, we have for all sets $A,B$ in $\INT{T}$,
$$P_{A}[\pi](p,q,  2g )\cap P_{B}[\pi](p,q,  2g)
=P_{A\cap B}[\pi](p,q, 2g).$$
We recall that the inclusion-exclusion formula asserts that any sets $S_t \subset S$ we have 
\begin{equation}\label{eq:sieve}
\left| S \backslash \bigcup_{ t = 1 }^T S_t\right| = |S| + \sum_{\emptyset \ne B \subset \INT{T}} (-1)^{|B|}  \left|  \bigcap_{ t \in B  } S_t\right| = |S| + \sum_{A \subsetneq \INT{T}} (-1)^{T - |A|}  \left|  \bigcap_{ t \notin A  } S_t\right|,
\end{equation}

We apply the inclusion-exclusion formula to the sets $S_t = P_{\INT{T}\backslash \{t\} } [\pi]  (p,q,2g) $ and $S = P (pq^{-1},2g)$. From what precedes
$$
\bigcap_{t \notin A}  P_{\INT{T}\backslash \{t\} }[\pi] (p,q,n)   = P_{\bigcap_{t \notin A } \INT{T}\backslash \{t\} }    [\pi] (p,q,n)=  P_{A}[\pi] (p,q,n).
$$
We thus have proved that 
$$P[\pi](p,q,2g)  =  \sum_{A\subset \INT{T}}(-1)^{T-|A|} |P_A[\pi](p,q,2g)|. $$
This completes the proof by plugging this last identity in Equation \eqref{eq:WgU} and Equation \eqref{eq:WgAex}. \end{proof}

Out of this, we are able to propose the key estimate for the centered Weingarten function. In the statement below, if $p$ and $q$ are 
two partitions of $\INT{k}$ then $p \vee q$ is the finest partition coarser than both $p$ and $q$.

\begin{theorem}\label{thm:centered-wg-estimate}
 Let $k$  even with $2 k^{7/2} \leq  n^{2}$, $\pi = (\pi_t)_{t \in \INT{T}}$ be a partition of $\INT{k}$, let $\veps \in \{\cdot,-\}^k$ be a 
 balanced sequence. For all $p, q \in P^\veps_k$, we have the following estimate: 
$$|\WgU[\pi](p,q ,n)|\le  (1 +  3 k^{7/2} n^{-2} )  n^{-k/2-|pq^{-1} |} 4^{|pq^{-1} |} (  k^{7/4}n^{-1                                                                                                                                                                                                                                                                                                                                                                                                                                                                                                                                                                                                                                                                                                                                                                                                                                                                                                                                                                                                                                                                                                                                                                                                                                                                                                                                                                                                                                                                                                                                                                                                                                                                                                                                                           })^{r},$$
where $r$ is the number of blocks of $\pi$ to which $p\vee q$ can be restricted.
\end{theorem}

Before we supply the proof, we give the main idea, which is quite simple: 
since in Theorem \ref{lemma:unitary-Wg-expansion3}, we realize Weingarten functions as unsigned sums, it is enough to estimate 
each summand separately, 
In turn, our estimate is quite blunt and relies solely on the inclusion
$P[\pi](p,q , l) \subset P(p,q ,l) $. In other words, an estimate
good enough for our purposes is achieved just because a partial connectedness condition
kills the first terms of a series. 

\begin{proof}[Proof of Theorem \ref{thm:centered-wg-estimate}]
Set $\sigma = pq ^{-1}$. From Theorem \ref{lemma:unitary-Wg-expansion3} we have
$$|\WgU [\pi](p,q,n)|= n^{-k/2-|\sigma|}  \sum_{g \ge 0} | P[\pi](p,q, 2g) |n^{-2g}.$$
We observe that $P[\pi](p,q , 2g) \subset P(p,q , 2g)$. We also claim that  $P[\pi](p,q , 2g)=\emptyset$ as soon as $2g < r $. Indeed, if 
$b$ is a block of $p \vee q$, then any solution of \eqref{hurwitz-centered} with $l=0$ satisfies $i_p,j_p$ is in $b$. Hence, if a block $\pi_t$ is a union of blocks of $p \vee q$, then at least two extra transpositions from an element in $\pi_t$ to another block have to be added to 
be 
an admissible solution of  \eqref{hurwitz-centered}. These two extra transpositions could be shared between two such blocks of $\pi$. 
It follows that  $P[\pi](p,q , 2g)\ne \emptyset$ implies $2g \geq r$ and 
$$|\WgU[\pi](p,q,n)| \le n^{-k/2-|\sigma |}  \sum_{g \ge r/2} |   P(p, q ,2g) | n^{-2g}.$$
The right-hand side can be estimated 
thanks to \eqref{eq:upper-bound-unitary} (applied to $k/2$). We get if $c = 3 / 2^{5/2}$ and $c k^{7/2}n^{-2} < 1$,
$$|\WgU[\pi](p,q,n)|\le n^{-k/2-|\sigma |} 4^{|\sigma |} \frac{ (c k^{7/2}n^{-2})^{ r/2}}{1-c k^{7/2}n^{-2}}.$$
If we assume further
$2 c k^{7/2}n^{-2}\le 1$, we obtain 
$$|\WgU[\pi](p , q ,n)|\le  \PAR{ 1 + \frac{4 c k^{7/2}}{n^2} } n^{-k/2-|\sigma |} 4^{|\sigma  |} (c k^{7/2}n^{-2})^{r/2},$$
as required (since $c \simeq 0.53$). \end{proof}

\subsection{From Weingarten calculus to Wick calculus}

\subsubsection{Case without brackets}

We start with
$x, y$ in $\INT{n}^k$ and  a balanced sequence $\varepsilon \in \{\cdot,-\}^k$.
If $U = (U_{ij})$ is Haar-distributed on $\Un$, we want to compare
$|\EE(U_{x_1y_1}^{\varepsilon_1}\ldots U_{x_ky_k}^{\varepsilon_k} )|$
with the matrix $U$ replaced by $G_{ij} / \sqrt n$, where $G_{ij}$ are independent complex standard Gaussian variables.  

We need a new definition. Let  $x,y \in \INT{n}^k$. If $u \in [n]$, we define the number of {\em left arms} of $u \in [n]$ in $(x,y)$ as 
$\sum_i \IND (x_i = u)$ and the number of {\em right arms} as $\sum_i \IND ( y_i = u)$. The pair $(x,y)$ is called an {\em even sequence} 
if for any $u \in [n]$, the number of left and right arms are even. 

Our result is as follows:

\begin{theorem}\label{theorem-warmup}
 Let $k$ be  even, $x, y$ in $\INT{n}^k$ and let $\veps \in \{\cdot,-\}^k$ be a balanced sequence. 
If $n \geq 4$ and $2 k^{7/2} \leq n^2$ then
\begin{equation*}\label{eq:prodtw}
n^{k/2} \left| \EE \PAR{ \prod_{i=1}^k U_{x_iy_i}^{\varepsilon_i}} \right|
\le \PAR{ 1 + 3 k^{7/2} n^{-2}} \EE \PAR{ \prod_{i=1}^k \PAR{G_{x_iy_i}^{\varepsilon_i}+ k n^{-1/4}}}.
\end{equation*}
Moreover, the above expectation on the left-hand side is zero unless $(x,y)$ is an even sequence.
\end{theorem}

\begin{proof}
From Theorem \ref{th:unitary-Wg}, we have 
\begin{equation}\label{eq:prodWgA}
n^{k/2} \EE \PAR{ \prod_{i=1}^k U_{x_iy_i}^{\varepsilon_i}}  = \sum_{p,q\in P^\veps_k}\delta_{p}(x)\delta_{q}(y) n^{k/2} \WgU (p,q,n),
\end{equation}
where we have identified a pair partition in $P^\veps_k$ with a bijection from the set $ \dot{\INT{k}} \subset \INT{k}$ of $i$'s such 
that $\veps_i = \cdot$ to the set of $i$'s such that $\veps_i = -$. The final statement of the theorem follows directly. 

Given a pair  $(p,q)$ involved in the right-hand side of \eqref{eq:prodWgA},
we introduce a subset $A$ of $\INT{k}$ which is the maximal subset such that
$p_{A}, q_{A}$, the restrictions of $p$ and $q$ to $A$, are well defined and describe the \emph{same} pair partition. In other words, 
$A$ is the union of all common pairs of $p$ and $q$. 
Note that $A$ could be empty or $\INT{k}$, but it has an even number of elements. 

Our strategy is to evaluate simultaneously all $(p,q)$ that yield the same $A$.
If we can find $\delta, \eta \geq 0$ such that for all $A \subset \INT{k}$, the sum in \eqref{eq:prodWgA} restricted to  pairs
$(p,q)$ that yield $A$ is bounded above by 
\begin{equation}\label{eq:wgA200}
( 1 + \delta ) \EE\PAR{\prod_{i\in A} G_{x_iy_i}^{\varepsilon_i}}\eta^{|A^c|}
\end{equation}
then we would deduce that 
$$n^{k/2} \left| \EE \PAR{ \prod_{i=1}^k U_{x_iy_i}^{\varepsilon_i}} \right|  \leq ( 1  + \delta) \EE   \PAR{ \prod_{i=1}^k \PAR{G_{x_iy_i}^{\varepsilon_i}+\eta }}
$$
since we have the following expansion: 
$$
\EE \PAR{ \prod_{i=1}^k \PAR{G_{x_iy_i}^{\varepsilon_i}+\eta }} = \sum_{A \subset \INT{k}} \EE \PAR{ \prod_{i\in A} G_{x_iy_i}^{\varepsilon_i}}\prod_{i\in A^c}\eta .
$$

To that end, we set $c = 3 / \sqrt 2$ and  $\sigma = pq^{-1}$ (seen as a bijection on $\dot{\INT{k}}$). 
According to Corollary \ref{cor-upper-bound-unitary} (applied to $k/2$), if  $c k ^{7/2} \leq n ^2$ and $\delta = 2 c k^{7/2} / n^2$, we have
$$n^{k/2 } \WgU  (p,q,n)\le \PAR{ 1 + \delta} \PAR{ \frac 4 n }^{|\sigma |}.$$

It is standard that if $\tau \in S_m$, then $|\tau | = m - \sum_l c_l$ where $c_l$ is the number of cycles of length $l$. 
Since $\sum l c_l = m \geq 2 (\sum_l c_l) - c_1$, we get $|\tau| \geq (m -c_1) /2 $. By assumption, all cycles of 
$\sigma$ in $A^c \cap \dot{\INT{k}}$ have length at least $2$, so the number of cycles of length $1$ of $\sigma$ is at most 
$|A \cap \dot{\INT{k}}| = |A|/2$. We deduce that
$$|\sigma | \geq (k   - |A| ) /4 = |A^c|/4.$$

We thus have proved that, if $n \geq 4$,  and $\eta_0 =( 4/n)^{1/4}$, 
$$n^{k/2} \WgU (p,q,n)\le  ( 1+ \delta) \eta_0^{|A^c|}.$$

Therefore, the  sum in \eqref{eq:prodWgA} restricted to pair partitions
$p,q$ that yield $A$ is upper bounded by 
$$(1+\delta) \eta_0^{|A^c|}    \PAR{(|A^c|/2)! }^2 \sum_{\tau \in P^\veps_{A}} \delta_\tau(x) \delta_\tau(y) , $$
where $P^\veps_A$ is the set of bijections on $A$ from the set $i$'s such that $\veps_i = \cdot$ to the set of $i$'s such that $\veps_i = -$. 
Moreover, the term $((|A^c|/2)! )^2$ accounts for the choices of $(p_{A^c}, q_{A^c})$.  From Wick formula \eqref{eq:WickC}, we have
$$ \EE\PAR{\prod_{i\in A} G_{x_iy_i}^{\varepsilon_i}}  =  \sum_{ \tau \in P^\veps_{A}} \delta_\tau(x) \delta_\tau(y) .$$
Finally, we recall that $m! \leq (m e^{-1})^m$. We deduce that \eqref{eq:wgA200} holds with $\delta$ as above and 
$\eta = \eta_0 e^{-1} (k/2)$. It concludes the proof.
\end{proof}

\subsubsection{Case with brackets}

We now move to the case with brackets. If $U = (U_{ij})$ is Haar-distributed on $\Un$, we want to compare expectations as in 
Equation \eqref{eq:prodbraex}
with the matrix $U$ replaced by $G_{ij} / \sqrt n$, where $G_{ij}$ are independent complex standard Gaussian variables. The main result 
in this direction is:
\begin{theorem}\label{theorem-with-brackets}
Let $k$  even with $2 k^{7/2} \leq  n^{2}$ and $n \geq 4$, $\pi = (\pi_t)_{t \in \INT{T}}$ be a partition of $\INT{k}$ such that each block has at most $\ell$ elements. Let $\veps \in \{\cdot,-\}^k$ be a balanced sequence. For any $x,y$ in  $\INT{n}^k$, we have
$$n^{k/2}\left| \EE \PAR{\prod_{t=1}^T \big[  \prod_{i \in \pi_t} U_{x_{i} y_{i}}^{\varepsilon_{i}} \big]} \right|\le (1 + \delta) \EE \PAR{\prod_{t=1}^T \PAR{ \big[  \prod_{i \in \pi_t} G_{x_{i} y_{i}}^{\varepsilon_{i}}   \big] + \eta }},$$
with $\delta = 3 k^{7/2} n^{-2}$ and $\eta  = 2 k^\ell n^{-1/4}$. Moreover, if each block $\pi_t$ 
contains an element with $\veps_i = \cdot$ and another with $\veps_i = -$, the same bound holds with $\eta = 2 k^\ell n^{-1/2}$. 
Finally, the above expectation on the left-hand side is zero unless $(x,y)$ is an even sequence.
\end{theorem}

We start by evaluating the average of products of brackets of shifted Gaussian variables.
\begin{lemma}\label{le:gaussiansieve}
 Let $k$ be  even, let $\veps \in \{\cdot,-\}^k$ be a balanced sequence and let $\pi = (\pi_t)_{t \in \INT{T}}$ be a partition of $\INT{k}$. 
 If $(g_i)_{i \in \INT{k}}$ is a complex Gaussian vector, then, for any complex number $\eta$, we have  
 $$
 \EE \PAR{\prod_{t=1}^T \big[  \prod_{i \in \pi_t} g_{i}^{\varepsilon_{i}}\big]} =  \sum_{p \in P^\veps (\pi)} \prod_{ (i,j) \in p} \dE ( g_i \bar g_{j} ),
 $$
where $P^\veps (\pi)$ is the set (possibly empty) of pair partitions on $\INT{k}$ such that all pairs $(i,j)$ satisfy:
(i)  $\veps_i = \cdot$, 
 $\veps_j = -$ 
 or the other way around,
 and (ii) for each block of $\pi$, there exists a pair $(i,j)$ in $p$  with one element in the block and the other outside. 
\end{lemma}

The proof is elementary and closely related to the proof of Theorem \ref{lemma:unitary-Wg-expansion3}; therefore, we omit it.

\begin{proof}[Proof of Theorem \ref{theorem-with-brackets}]
We follow the same strategy as in the proof of Theorem \ref{theorem-warmup}. By Proposition\;\ref{weingarten-centered}, 
\begin{equation}\label{eq:prodWgU2}
n^{k/2}  \EE \PAR{\prod_{t=1}^T \big[  \prod_{i \in \pi_t} U_{x_{i} y_{i}}^{\varepsilon_{i}} \big]}   = \sum_{p,q\in P^\veps_k}\delta_{p}(x)\delta_{q}(y) n^{k/2} \WgU[\pi] (p,q,n),
\end{equation}
where as usual we have identified a pair partition in $P^\veps_k$ with a bijection from the set $ \dot{\INT{k}} \subset \INT{k}$ of $i$'s 
such that $\veps_i = \cdot$ to the set of $i$'s such that $\veps_i = -$.  The last statement of the theorem follows directly.

On the other hand, for $B \subset \INT{T}$, we set $\INT{k}_B = \cup_{t \in B} \pi_t$.  We define  $P^\veps_B$ as the set of pair partitions on 
$\INT{k}_B$ whose pairs $(i,j)$ are such that $\veps_i = \cdot$, $\veps_j = -$ and we let  $P^\veps_B (\pi) \subset P^\veps_B$ be the pair partitions 
such that for each block of $\pi$, there exists a pair $(i,j)$ with one element in the block and the other outside. By Lemma \ref{le:gaussiansieve}, 
we have, for any numbers $\eta_1,\eta_2$,
\begin{eqnarray}
\EE  \PAR{\prod_{t \in T} \PAR{ \big[  \prod_{i \in \pi_t} (G_{x_iy_i}^{\varepsilon_{i}}  ) \big] + \eta_1 + \eta_2 }} & =&  \sum_{A \subset \INT{T}} \eta_2^{|A|^c} \sum_{B \subset A} \eta_1^{|B^c|} \sum_{\tau \in P^\veps_B (\pi)} \EE  \PAR{\prod_{t \in B} \big[  \prod_{i \in \pi_t} (G_{x_iy_i}^{\varepsilon_{i}}  ) \big] }\nonumber \\
& = &  \sum_{A \subset \INT{T}} \eta_2^{|A|^c} \sum_{B \subset A} \eta_1^{|B^c|} \sum_{\tau \in P^\veps_B (\pi)} \delta_\tau(x)\delta_\tau(y),\label{eq:prodbrag}
\end{eqnarray}
where  $B^c = A \backslash B$ is the complement of $B$ in  $A$. 

Given a pair $(p,q)$ involved in the right-hand side  of \eqref{eq:prodWgU2}, we introduce a subset $A$ of $\INT{T}$ which is the 
maximal subset such that
$p_{A}, q_{A}$, the restrictions of $p$ and $q$ to $\INT{k}_A$, are well defined and describe the \emph{same} pair partition, say $\tau \in P^\veps_A$. 
According to Theorem \ref{thm:centered-wg-estimate}, if $\sigma = pq^{-1}$,  we have 
$$
n^{k/2}  |\WgU[\pi](p,q ,n)|\le  \PAR{ 1 +  \delta } \PAR{\frac{4}{n}}^{|\sigma|}   \eta^{r},
$$
with $\delta = 3 k^{7/2} n^{-2}$, $\eta=k^{7/4}n^{-1}$ and $r$ is the number of blocks of $\pi$ to which $p \vee q$ can be restricted. 
The number of cycles of $\sigma$ of length $1$ is at most $k/2 - |A^c|/2$ (remark: it is at most $k/2 - |A^c|$ if each block contains 
an element with $\veps_i = \cdot$ and another with $\veps_i = -$). As in the proof of Theorem \ref{theorem-warmup}, we deduce that 
$|\sigma| \geq |A|^c / 4$ and, if $n \geq 4$, 
$$
n^{k/2}  |\WgU[\pi](p,q ,n)|\le  \PAR{ 1 +  \delta } \eta_0^{|A^c|}   \eta^{r},
$$
with $\eta_0  = (4/n)^{1/4}$. We note also that $r$ is additive over the restrictions of $\pi$ to $\INT{k}_A$ and $\INT{k}_{A^c}$. 
Since there are at most $((|\INT{k}_{A^c}|/2) !)^2$ choices for the restrictions of $p$ and $q$ to $\INT{k}_{A^c}$, the contribution in 
\eqref{eq:prodWgU2} of the sum over all pairs $(p,q)$ which yields the same set $A$ is thus upper bounded by 
\begin{equation}\label{eq:prodWgU2A}
\PAR{ 1 +  \delta } \eta_0^{|A^c|} ((|\INT{k}_{A^c}|/2) !)^2   \sum_{\tau \in P^{\veps}_{A}} \eta^{r(\tau)} \delta_{\tau}(x)\delta_{\tau}(y),
\end{equation}
where  $r(\tau)$ 
is the number of blocks of $\pi$ which is a union of pairs of $\tau$.

By assumption $|\INT{k}_{A^c}| \leq \ell |A|^c$. Using $m! \leq (m e^{-1})^m$, we deduce that the expression \eqref{eq:prodWgU2A} 
is upper bounded by 
\begin{equation}\label{eq:prodWgU2A2}
\PAR{ 1 +  \delta } \eta_2^{|A^c|}  \sum_{\tau \in P^{\veps}_{A}} \eta^{r(\tau)} \delta_{\tau}(x)\delta_{\tau}(y),
\end{equation}
with $\eta_2 = k^\ell n ^{-1/4}$. We now take a closer look at the sum in \eqref{eq:prodWgU2A2}. Let us call $B\subset A$ 
the complement in  $A$ of the blocks of $\pi$ which are union of pairs of $\tau$. We have $r(\tau) = |B^c|$ where 
$B^c = A \backslash B$. For a given set $B \subset \INT{k}_A$, let us call $P^\veps_{A,B}$ the subset  
$P^{\veps}_{A}$ which yields the set $B$. By assumption, the restriction of $\tau$ in $P^\veps_{A,B}$ to $B$ is a pair partition in 
$P_B^\veps(\pi)$. Conversely,  recall that the blocks of $\pi$ have at most $\ell$ elements. For a given pair partition $\tau'$ in $P_B^\veps(\pi)$, 
there are at most $(\ell-1)^{|B^c|/2}$ partitions in  
$P^\veps_{A,B}$  whose restriction to $B$ is $\tau'$. If $\eta_1 = \ell^{\ell/2} \eta$, we thus have proved that 
$$
\sum_{\tau \in P^{\veps}_{A}} \eta^{r(\tau)} \delta_{\tau}(x)\delta_{\tau}(y) \leq \sum_{B \subset  A} \eta_1^{|B^c|} \sum_{\tau \in P^\veps_B (\pi)} \delta_\tau(x)\delta_\tau(y).
$$
Note also that $\eta_1 \leq \eta_2$ for our choice of $k$. This concludes the proof in view of Equation \eqref{eq:prodbrag}. \end{proof}

\begin{remark}\normalfont
Note in passing that the lower bound in Proposition \ref{combinatorial-estimate} can be used to show that there is some symmetry 
in Theorem \ref{theorem-warmup} and Theorem
 \ref{theorem-with-brackets}. Namely, it is possible in these results to swap the roles of the 
 Gaussian entries and of the unitary entries (the values of the constants $\delta,\eta$ need to be adjusted). This is just an esthetic comment about the sharpness of the comparison, as
 we need the bound, as stated in the Theorem
 \ref{theorem-with-brackets}.
 \end{remark}

\subsection{Moment bounds for a product of unitaries}

We conclude this section with a corollary of Theorem \ref{theorem-with-brackets}. Let $x,y$ be two sequences in $\INT{n}^k$. 
 The {\em multiplicity}  of $e = (a,b) \in \INT{n}^2$ is defined as $\sum_{i} \IND ( (x_{i},y_{i} ) = e )$. The set of pairs of multiplicity at 
 least one is the set of visited pairs $\cup_i \{ (x_i,y_i) \}$. Moreover if $\pi = (\pi_t)_{t \in \INT{T}}$ is a partition of $\INT{k}$, we say that 
 $\pi_t$ is an {\em isolated} block of $(x,y)$ if for all $i \in \pi_t$, for all $(x_{i},y_{i}) \ne (x_{j}, y_{j})$ for all 
 $j \in \INT{k} \backslash \pi_t$ (in other words, $(x_{i},y_{i})$ is of multiplicity $0$ in the sequence 
 $(x_{j},y_{j})_{j \in \INT{k} \backslash \pi_t}$).

\begin{corollary}\label{cor:WG2}
 Let $k = q T$  even with $k^{q+1} \leq  n^{1/4}$, $\pi = (\pi_t)_{t \in \INT{T}}$ be a partition of $\INT{k}$ such that each block has  at least
 $q$ elements. Let $\veps \in \{\cdot,-\}^k$ be a balanced sequence. For any $x,y$ in  $\INT{n}^k$, we have, for some universal 
 constant $c>0$, 
$$\left| \EE \PAR{\prod_{t=1}^T \big[  \prod_{i \in \pi_t} U_{x_{i} y_{i}}^{\varepsilon_{i}} \big]} \right| \leq c n^{-\frac{k}{2}}  \eta^{  b + \frac {e_1} q } k^{\frac{m_4}{ 2}},$$
where $\eta = c k^{q/2} n^{-1/8}$, $e_1$ is the number of pairs of multiplicity $1$, $b$ is the number of isolated $(x_t,y_t)$ and 
$m_{4}$ is the sum of multiplicities of pairs with multiplicity at least $4$.
\end{corollary}
\begin{proof}
Let $\eta,\delta$ be as in Theorem \ref{theorem-with-brackets} with $\ell =q$. From \eqref{eq:prodbrag}, we have
\begin{equation}\label{eq:corWG2}
 I :=(1+\delta)^{-1} n^{k/2}\left| \EE \PAR{\prod_{t=1}^T \big[  \prod_{i \in \pi_t} U_{x_{i} y_{i}}^{\varepsilon_{i}} \big]} \right| \leq \sum_{A \subset \INT{T}} \eta^{|A|^c}  \sum_{\tau \in P^\veps_A (\pi)} \delta_\tau(x)\delta_\tau(y).
\end{equation}
From the definition of $P^\veps_A(\pi)$, we have $\delta_\tau(x)\delta_\tau(y)  = 0$ if $A$ contains a pair of odd multiplicity or if 
$A$ intersects an isolated block of $(x,y)$. We set $\INT{k}_A = \cup_{t \in A} \pi_t$. If $(G_{ij})$ and $Z$ are independent standard complex Gaussian variables,
$$
\sum_{\tau \in P^\veps_A (\pi)} \delta_\tau(x)\delta_\tau(y) \leq \sum_{\tau \in P^\veps_A} \delta_\tau(x)\delta_\tau(y) = \EE    \prod_{i  \in \INT{k}_A} G^{\veps_i}_{x_{i} y_{i}} = \prod_{(a,b) \in \INT{n}^2} \dE [ Z^{m^ {\cdot}_{ab}(A)} \bar Z^{ m^{\bar{}}_{ab}(A)}],
$$
where $m^\veps_{ab}(A)$  is the number of times that $(x_i,y_i,\veps_i) = (a,b,\veps)$ for $i \in \INT{k}_A$. If $m_{ab}$ is the multiplicity of 
$(a,b)$ in $(x,y)$, we deduce that 
$$
\EE    \prod_{i  \in \INT{k}_A} G^{\veps_i}_{x_{i} y_{i}}  \leq \prod_{(a,b) \in \INT{n}^2} \dE  |Z|^{m_{ab}},
$$
From Wick formula, for even $m$, $\dE  |Z|^{m} = (m/2) ! \leq m^{m/2}$ and $\dE |Z|^2 =1$. We deduce that
$$
\EE    \prod_{i \in \INT{k}_A} G^{\veps_i}_{x_{i} y_{i}}  \leq k^{m_4/2}.
$$

Let $T_0$ be set of $t \in T$ such that $\pi_t$ is isolated or contains a pair $(x_i,y_i)$ of multiplicity one.  We set $t_0 = |T_0|$. 
From \eqref{eq:corWG2}, we find
\begin{eqnarray*}
I & \leq & k^{m_4/2} \sum_{A \subset \INT{T}, A \cap T_0 = \emptyset} \eta^{|A^c|}  \\
& \leq & k^{m_4/2}    \sum_{s = 0}^{T-t_0} 
{T -t_0 \choose s} \eta^{t_0+s}  \\
 &= & k^{m_4/2}  \eta^{t_0} ( 1+ \eta)^{T - t_0}, 
\end{eqnarray*}
where we have upper bounded all possibilities of sets $A^c$ of size $t_0  +s$ in terms of its intersections with the set 
$T_0$ and its complement. We get, 
$$
I \leq  k^{m_4/2}   (1+ \eta)^{T} \eta^{t_0}.
$$
By assumption, $t_0  \geq \max(b , e_1/q) \geq (b + e_1/q)/2$. The statement of the corollary follows by using that, for $u,v \geq 0$,
$(1 + u)^v \leq e^{uv} \leq 1 + e^c uv$ if $uv  \leq c$. 
\end{proof}

\begin{remark}\normalfont Corollary \ref{cor:WG2} is tailored for our needs. This is not an optimal consequence of Theorem 
\ref{theorem-with-brackets}. Along the lines of the proof of Corollary \ref{cor:WG2}, it is possible to obtain sharper bounds, for example, 
by using the fact that the odd moments of a Gaussian random variable vanish. 
\end{remark}

\subsection{The orthogonal case}
\label{subsec:orthogonal-variations}

This paper focuses on random unitary matrices; however, as we claim in the introduction, the obvious variant of our results 
works precisely the same way for sequences of orthogonal groups. This subsection outlines how to adapt the above statements on the unitary group to the orthogonal case.
Firstly, there exists a counterpart of Proposition \ref{weingarten-centered}, which can be stated as follows:

\begin{proposition}\label{weingarten-centered-orthogonal}
Let $k$ be  even, $\pi = (\pi_t)_{t \in \INT{T}}$ be a partition of $\INT{k}$, and $x, y$ in $\INT{n}^k$.  
There exists a generalized Weingarten function
$\WgU_O[\pi] (p,q,n)$ such that
\begin{equation*}
\EE \PAR{\prod_{t=1}^T \big[  \prod_{i \in \pi_t} O_{x_{i} y_{i}} \big]}
=\sum_{p,q\in P_k}\delta_{p}(x)\delta_{q}(y)\WgU_O[\pi] (p,q,n).
\end{equation*}
\end{proposition}

Intriguingly, this formula seems to be simpler. Indeed, there is no need to consider conjugates, specify matchings, or balancing
conditions since the entries $O_{ij}$ of a random orthogonal matrix are all real. Its proof is based on the existence of the orthogonal Weingarten 
function and then repeats the argument of Proposition \ref{weingarten-centered}.

The subsequent estimate is the orthogonal counterpart of 
Theorem \ref{thm:centered-wg-estimate}.

\begin{theorem}\label{thm:centered-wg-estimate-orthogonal}
 Let $k$  even with $4 k^{7/2} \leq  n^{2}$, $\pi = (\pi_t)_{t \in \INT{T}}$ be a partition of $\INT{k}$. For all $p, q \in P_k$, the following estimate holds true: 
$$|\WgU_O[\pi](p,q ,n)|\le  (1 +  12 k^{7/2} n^{-1} )  n^{-k/2-|pq^{-1} |} 4^{|pq^{-1} |} ( 2 k^{7/4}n^{-1                                                                                                                                                                                                                                                                                                                                                                                                                                                                                                                                                                                                                                                                                                                                                                                                                                                                                                                                                                                                                                                                                                                                                                                                                                                                                                                                                                                                                                                                                                                                                                                                                                                                                                                                                           })^{r},$$
where $r$ is the number of blocks of  $\pi$ to which $p\vee q$ can be restricted. 
\end{theorem}

Without entering details, the proof is essentially the same as the proof of Theorem \ref{thm:centered-wg-estimate}. The main input 
is an analogue of Theorem \ref{lemma:unitary-Wg-expansion1} and Proposition \ref{combinatorial-estimate} given in  \cite{MR3680193} by 
Theorem 4.6 and Theorem 4.9. We note that in \cite{MR3680193}, these two results are slightly more difficult to prove than their unitary counterpart. 
This yields an orthogonal version of Theorem \ref{theorem-with-brackets} and
Corollary \ref{cor:WG2}.

\section{Strong asymptotic freeness through non-backtracking operators}
\label{sec:SAFNB}

This section aims to extend \cite[Section 3]{MR4024563} to general bounded operators in Hilbert spaces. 

\subsection{Spectral mapping formulas}

We consider $(b_1, \ldots, b_\ell)$  elements in $\mathcal B(\mathcal H)$ where $\mathcal H$ is a Hilbert space. We assume that the set 
$\INT{\ell}$ is endowed with an involution $i \mapsto i^{*}$. 

The non-backtracking operator associated with the $\ell$-tuple of matrices 
$(b_1, \ldots, b_\ell)$ 
 is the operator on 
$\mathcal B(\mathcal H \otimes \dC^\ell)$ defined by
\begin{equation}
\label{eq:defBB}
B = \sum_{j \ne i^*} b_j  \otimes E_{ij},
\end{equation}
where  $E_{ij} \in M_\ell( \dR)$ are the canonical matrix elements. We also define the left non-backtracking operator as 
\begin{equation}
\label{eq:defBBl}
\widetilde B = \sum_{j \ne i^*} b_i  \otimes E_{ij}.
\end{equation}
Note that if the $b_i$'s are invertible, then $B$ and $\widetilde B$ are conjugate
$$
\widetilde B = D B D^{-1}
$$
with $D = \sum_i b_i \otimes E_{ii}$. The non-backtracking operators are used in this paper to give an alternative description of the 
spectrum of an operator of the form \eqref{eq:defA}.  We start with a result in the reverse direction. In the sequel for shorter notation, 
the identity operator in $\mathcal H$ is denoted by $1$, and if $a$ is in $\mathcal B(\mathcal H)$ and $\lambda \in \dC$, we write 
$a- \lambda$ in place of $a - \lambda 1$.

\begin{proposition}\label{prop:nonbackgeneral}
Let $B$ be as above and let  $\lambda \in \dC$ satisfy 
$\lambda^2  \notin  \{ \si(  b_i b_{i^*} ): i \in \INT{\ell} \}$. Define the operator $A^{(\lambda)}$ on $\mathcal H$  through
\begin{equation} \label{def_AM0}
A^{(\lambda)}  = b_0 ( \lambda) +  \sum_{i=1}^\ell b_i (\lambda) \,, \qquad b_i (\lambda) = \lambda  b_{i} ( \lambda^2 - b_{i^*} b_{i} )^{-1}   \AND b_0 ( \lambda) =  - 1 -  \sum_{i=1}^\ell  b_{i}(  \lambda^2 - b_{i^*} b_i )^{-1} b_{i^*}  .
\end{equation} 
Then $\lambda \in \sigma (B)$ if and only if  $0 \in \sigma ( A^{(\lambda)})$. Similarly, $\lambda \in \sigma (\widetilde B)$ if and only if  
$0 \in \sigma ( A^{(\lambda)})$. 
\end{proposition}

\begin{proof}  Let us assume that $\lambda$ is in the point spectrum of $B$, i.e. there is a non-zero vector 
$v  \in  \mathcal H \otimes \dC^\ell$ such that $B v =  \lambda v$. Let $E_i, i\in \INT{\ell}$ be the canonical basis of $\dC^\ell$.
We can write $v=\sum_{i\ }v_i\otimes E_i$.
We define a vector $u\in \mathcal H$ by
\begin{equation}\label{000}
u=\sum_{j } b_j  v_j.
\end{equation}
Let us show that 
$u$ is a non-zero vector $u\in  \mathcal H$ such that $A^{(\lambda)}  u=0$.
Component-wise, the equation $B   v =  \lambda v$ can be written
\begin{equation}\label{111}
\lambda v_i= \sum_{j \ne i^*} b_j v_j=
 u-b_{i^*} v_{i^*}.
\end{equation}
In the above equation, if we replace $i$ by $i^*$ 
we get
\begin{equation}\label{222}
\lambda v_{i^*}=u-b_i v_i.
\end{equation}
Multiplying Equation \eqref{111} by $\lambda$ and substituting the last term of its right-hand side with
Equation \eqref{222} multiplied by $b_{i^*}$, we get
$$\lambda^2 v_i=
\lambda u-b_{i^*}  u+b_{i^*}b_i v_i .$$
This can be rewritten as
$$(\lambda^2 -b_{i^*}b_i) v_i= (\lambda   -b_{i^*}) u .$$
Since $\lambda^2-a_{i^*}a_i$ is invertible, we get, for all $i\in \INT{\ell}$,
\begin{equation}\label{333}
v_i=(\lambda^2  -b_{i^*}b_i)^{-1}(\lambda  -b_{i^*})  u .
\end{equation}
If $u$ were the zero vector, then all $v_i$ would be zero. Therefore $v$ would be zero, which contradicts that $v$ is an eigenvector. 
This proves that $u$ is not zero. 

Let us now prove that $A^{(\lambda)} u=0$, i.e. its kernel is nontrivial. 
According to the definition of $A^{(\lambda)}$, we need to prove that
$$u=\sum_{i}  \PAR{ -b_{i}(  \lambda^2  - b_{i^*} b_i )^{-1} b_{i^*} + 
\lambda  b_{i} ( \lambda^2 - b_{i^*} b_{i} )^{-1} } u.$$
The right-hand side is equal to
$$\sum_{i} 
b_{i}
 (\lambda^2 -b_{i^*}b_i)^{-1}(\lambda  -b_{i^*}) u.$$
Substituting with the help of Equation \eqref{333} for each $i$, this is equal to
$\sum_{i} 
b_{i} v_i,$
which completes the claim from the definition of $u$ in \eqref{000}.

Conversely, if $0$ is in the point spectrum of $A^{(\lambda)} $ with eigenvector $u$,
we define $v_i$ with Equation \eqref{333}. However, the above computations imply that $Bv =  \lambda v$ and
$u=\sum_{j}b_jv_j$ as per the original definition of $u$, therefore $u$ would be zero, 
if $v$ is zero 
which contradicts the assumption.  We thus have proved so far that $0$ is in the discrete spectrum of $A^{(\lambda)}$ if and only if 
$\lambda$ is in the point spectrum of $B$.

The same equivalence holds for $\widetilde B$. Indeed, if $\widetilde B v = \lambda v$ then we define $u = \sum_{i } v_i$. Repeating 
the above computation, we find $v_i  = b_i ( \lambda^2  - b_i b_{i^*} )^{-1}( \lambda - b_i) u$ and $A^{(\lambda)} u = 0$. Conversely, 
if $A^{(\lambda)} u = 0$, then we define $ v = \sum_i v_i \otimes E_i$ with $v_i =  a_i ( \lambda^2  - a_i a_{i^*} )^{-1} ( \lambda - a_i) u $. 
We get $Bv = \lambda v$. Arguing as above, the equivalence follows.

Next, we handle the continuous spectrum. 
If $\lambda$ is in the continuous spectrum of $B$, then 
there exists a sequence of unit vectors $v^{(n)}\in \mathcal H \otimes \dC^\ell$ such that
$\|\lambda v^{(n)}-B v^{(n)}\|_2\to 0$ (see for example Kowalski \cite[p. 19]{kowalskiLN}) .
To each $v^{(n)}$ we associate a vector $u^{(n)}\in \mathcal H$
thanks to Equation \eqref{000}.
In the first part of the proof, all equations are continuous.
Specifically, Equations \eqref{111}-\eqref{333} remain correct if the right-hand side is perturbed additively 
by a vector $\varepsilon_n$
whose norm goes to $0$ as $n\to\infty$.
This implies that $\|A^{(\lambda)}  u^{(n)}\|_2\to 0$, i.e. $0$ is in the spectrum of $A^{(\lambda)}$.

Conversely, if $0$ is in the continuous spectrum of $A^{(\lambda)}$, let $u^{(n)}\in \mathcal H \otimes  \dC^\ell$ such that
$\|A^{(\lambda)} u^{(n)}\|_2\to 0$  as $n\to\infty$.
We define $v^{(n)}$ through equation \eqref{333}. As above, it is possible to prove that 
$\|\lambda v^{(n)}-B v^{(n)}\|_2\to 0$ as $n\to\infty$, so $\lambda$ is in the spectrum of $B$.

Finally, we handle the residual spectrum. Let $v \in \mathcal H \otimes \dC^\ell$ non-zero that is not an element in the closure of the image of 
$B - \lambda 1$, and without loss of generality
let us assume that $v$ is orthogonal to the image of $B-\lambda 1$.  Then, for all 
$x \in \mathcal H \otimes \dC^\ell$, $\langle Bx-\lambda x,v\rangle =0$.
This implies that $B^* v=\overline\lambda v$, that is, $\overline\lambda$ is in the point spectrum of $B^*$. Now, observe that 
$$
B^* = \sum_{i \ne j ^*} b^*_{j} \otimes E_{ji} = \sum_{i \ne j ^*} b^*_{i} \otimes E_{ij} = \sum_{i \ne j ^*} \tilde b_i \otimes E_{ij},
$$
with $\tilde b_i = b_i^*$. In particular, $B^*$ is a left non-backtracking operator as defined in \eqref{eq:defBBl} with underlying operators 
the $\tilde b_i$'s. From what precedes, $\overline\lambda$ is in the point spectrum of $B^*$ if and only if $0$ is in the point spectrum 
of $\widetilde A^{(\lambda)} = \tilde b_0  +  \sum_i \tilde b_i  $ defined by, for $i \in \INT{\ell}$,
$$
\tilde b_i   = \overline \lambda b^*_{i} ( \overline \lambda^2 - b^*_{i^*} b^*_{i} )^{-1}   \AND \tilde b_0  =  - 1 -  \sum_{i=1}^\ell  b^*_{i}( \overline  \lambda^2 - b^*_{i^*} b^*_i )^{-1} b^*_{i^*}  .
$$
It straightforward to check that $\tilde b_i = b_i (\lambda) ^*$ for each $i \in \INT{\ell}$. Thus, we have 
$\widetilde A^{(\lambda)} = A^{(\lambda)*} $. From what precedes, we get that $\overline\lambda$ is in the discrete spectrum of $B^*$ 
if and only if $0$ is in the discrete spectrum of $A^{(\lambda)*}$. Hence, we have proved that if $\lambda$ is in the residual spectrum 
of $B$, then $0$ is in the spectrum of $A^{(\lambda)}$ (since for any bounded operator $T$ on a Hilbert space, $\sigma(T^*) $ is the 
complex conjugate of the set $\sigma(T)$, see Reed-Simon \cite[Theorem VI.7]{MR0493421}).

Conversely, if $0$ is in the residual spectrum of $A^{(\lambda)}$, then $0$ is in the discrete spectrum of $A^{(\lambda)*}$. From what 
precedes, we get that $\overline \lambda$ is in the discrete spectrum of $B^*$. In particular, $\lambda$ is in the spectrum of $B$. 
The proposition is proved.\end{proof}

We now consider the situation where the Hilbert space $\mathcal H$ is of the form $\dC^r \otimes \mathcal K$, where $\mathcal K$ 
is a Hilbert space and  
$$
b_i=a_i\otimes V_i
$$
with $a_i \in M_r(\dC)$ and $V_i$ unitary operator on $\mathcal K$ such that for all $i \in \INT{\ell}$,
$$
V_{i^*} = V_i^*.
$$
Moreover, we assume that $\ell = 2 d$ and the involution $i^*$ is as in Section \ref{sec:main}. In this specific case, we have 
\begin{equation}\label{eq:defB0}
B = \sum_{i \ne j^*} a_j \otimes V_j \otimes E_{ij}  \AND \widetilde B = \sum_{i \ne j^*} a_i \otimes V_i \otimes E_{ij}.
\end{equation}
The operator $A^{(\lambda)}$ in \eqref{def_AM0} is given by 
\begin{equation} \label{def_AMU}
A^{(\lambda)}  = a_0 ( \lambda)  \otimes  1 +  \sum_{i=1}^{2d} a_i (\lambda) \otimes V_i
\end{equation}
with 
$$\qquad a_i (\lambda) = \lambda  a_{i} ( \lambda^2 - a_{i^*} a_{i} )^{-1}   \AND a_0 ( \lambda) =  - 1 -  \sum_{i=1}^{2d} a_{i}(  \lambda^2 - a_{i^*} a_i )^{-1} a_{i^*}.
$$

If $\mathcal K=\ell^2(X)$ with $X$ countable and $V_i$ a permutation operator, 
then Proposition \ref{prop:nonbackgeneral} recovers Proposition 9 of \cite{MR4024563} up to the minor point that
$B$ is slightly modified into $B = \sum_{j \ne i^*}   a_j\otimes S_i  \otimes E_{ij}$, but it is easy to check that
both $B$ are conjugate to each other, so they have the same spectrum. 
It follows that all results of  \cite[Section 3]{MR4024563}  can be extended to this more general setting. 

We now review two key results \cite[Section 3]{MR4024563} that will be used in the sequel. We take $\ell = 2d$. We start with a kind of converse of 
Proposition \ref{prop:nonbackgeneral}, in the sense that for a given self-adjoint operator $A$ of the form 
\begin{equation}\label{eq:defA0}
A = a_0 \otimes 1 + \sum_{i=1}^{2d} a_i \otimes V_i,
\end{equation}
satisfying the symmetry condition \eqref{eq:symA}  and a real number $\mu$, we look for a non-backtracking operator $B_\mu$ which detects if $\mu \in \sigma(A)$. To perform this, we need to introduce the resolvent of the operator $A_\star$ defined in \eqref{eq:defAfree}. For 
$\mu \notin \sigma( A_\star)$, we set
$$
G(\mu) = ( \mu - A_\star  )^{-1}.
$$
For $x,y\in \Fd$, we define $G_{xy} (\mu)  \in M_r(\dC)$ as the matrix $P_x G(\mu) P^*_y$ where $P_x$ is the projection onto the vector 
space $\dC^r \otimes \delta_x$. We denote by $o$ the neutral element of $\Fd$ for its group structure and $g_i$ the $i$-th generator of 
$\Fd$.  Finally, if $D$ is a bounded set in $\dC$, the convex hull of $D$ is denoted by $\hull( D ) $.

\begin{proposition}\label{prop:nonback2}
Let $A$ be as in \eqref{eq:defA0} satisfying \eqref{eq:symA} and $\mu \notin \hull( \sigma ( A_\star))$. Define the matrices for $i \in \INT{2d}$,
\begin{equation*} \label{def_AM}
\hat a_i (\mu) = G_{oo}  ( \mu  )^{-1} G_{o g_i} ( \mu ). 
\end{equation*} 
Let $B_\mu =  \sum_{i \ne j^*} \hat a_j (\mu)\otimes V_j \otimes E_{ij} $ be the corresponding non-backtracking operator. Then 
$\mu \notin \sigma (A)$ if and only if  
$1  \notin \sigma ( B_\mu)$. The same statement holds for the corresponding left non-backtracking operator $\widetilde B_\mu$.
\end{proposition}

This statement is \cite[Proposition 10]{MR4024563} extended for a general Hilbert space and general unitaries $V_i$'s. The same proof 
applies in this case thanks to Proposition \ref{prop:nonbackgeneral}.

\subsection{Spectral radius of non-backtracking operators}

We conclude this section with a sharp criterion to guarantee in terms of non-backtracking 
operators that the spectrum of an operator $A$ is in a neighborhood of the spectrum of the operator $A_\star$. The following result 
is \cite[Theorem 12]{MR4024563}. Again,  thanks to the improvement of Proposition \ref{prop:nonbackgeneral}, we can now state it in a 
more general context.

\begin{theorem}\label{prop:edgeAB}
Let $A$ be as in \eqref{eq:defA0} satisfying \eqref{eq:symA} and $A_\star$ the corresponding free operator defined by 
\eqref{eq:defAfree}. The following two results hold:
\begin{enumerate}[(i)]
\item
For any $\mu \notin \hull ( \sigma (A_\star) ) $, we have $\rho( (B_\star)_\mu ) < 1$, 
where $(B_\star)_\mu  =  \sum_{i \ne j^*} \hat a_j (\mu)\otimes \lambda (g_j) \otimes E_{ij} $ is the non-backtracking operator on $\Fd$ 
with weights as in Proposition \ref{prop:nonback2}.
\item
For any $\veps >0$, there exists $\delta > 0$ such that if for all real $\mu$ at distance at least $\veps$ from $\hull( \sigma (A_\star) )$,
$$
\rho(B_\mu) < \rho( (B_\star)_\mu ) + \delta, 
$$
then $\hull(\sigma(A))$ is in an $\veps$-neighbourhood of $\hull( \sigma (A_\star) )$. 
\end{enumerate}
Moreover, 
the same holds for the left non-backtracking operators. 
\end{theorem}

The theorem is a simple consequence of the following two claims: 
$$
\hbox{for any $(b_1,\ldots,b_{2d})$ in $M_r(\dC)^{2d}$, } \quad  \rho(B_\star) = \sup \{ \lambda \geq 0 : \lambda \in \sigma(B_\star) \},
$$
and 
\begin{equation}\label{eq:src}
\hbox{the map from $M_r(\dC)^{2d}$ to $\dR$:  $(b_1,\ldots,b_{2d}) \mapsto \rho(B_\star)$ is continuous.}
\end{equation} 
We refer to \cite{MR4024563} for details.

\section{Expected high trace of non-backtracking matrices}
\label{sec:trace}

In this section, we prove  Theorem \ref{th:main} by computing the spectral radius of the non-backtracking matrices associated to the 
random matrix $A$ defined in \eqref{eq:defA}. 

\subsection{Main technical statement}

We now come back to the setting of Theorem \ref{th:main}. Let $U_1, \ldots,U_d$ be independent Haar-distributed random unitary 
matrices in $\Un$.  We define $V_i$ as in \eqref{eq:defVi} and its centered version $[V_i]$ as in \eqref{eq:bracket}. 

The left non-backtracking operator associated to the weights $(a_1 \otimes [V_1],\ldots,a_{2d}\otimes [V_{2d}])$ is 
\begin{equation}\label{eq:defBV2}
B = \sum_{i \ne j^*} a_i \otimes [V_i] \otimes E_{ij}.
\end{equation}
Note that we have omitted the tilde in \eqref{eq:defBBl} for shorter notation. The main technical result of this section is an upper bound 
on the spectral radius  of $B$ in terms of the spectral radius of the corresponding operator on the free group $\Fd$:
\begin{equation}\label{eq:defBstar}
B_\star = \sum_{i\ne j^*} a_i \otimes \lambda(g_i) \otimes E_{ij},
\end{equation}
where $\lambda(g)$ is the left-regular representation of $g \in \Fd$.  Let $\rho (B_\star)$ be the spectral radius of the $B_\star$. 
Recall that $\rho(B) \leq \| B^\ell \|^{1/\ell}$ for all integer $\ell \geq 1$ and the sequence $\| B^\ell \|$ is sub-multiplicative in $\ell$.

\begin{theorem}\label{th:FKB}
There exists a universal constant $c >0$ such that if $q \leq c \ln n /  \ln \ln n $, then the following holds. 
For any $\veps >0$, there exists a constant $C \geq 1$ (depending on $\veps$, $d$ and $r$) such that for any 
$(a_1,\ldots,a_{2d}) \in  M_r(\dC)$ with $\max_i \| a_i \| \leq 1$, if $\ell = \lfloor C q \rfloor$, the event
$$
 \| B^\ell \|^{1 / \ell} \leq \rho (B_\star) + \veps
$$
holds with probability at least $ 1 - C \exp \PAR{-(\ln n)^2}$.
\end{theorem}

In the following subsection, we prove Theorem \ref{th:FKB}. In the last subsection, we deduce Theorem \ref{th:main} from 
Theorem \ref{th:FKB}.

\subsection{Proof of Theorem \ref{th:FKB}}

 In the sequel,  an entry of the matrix $[V_i]$ will be denoted by 
$$[V_i]_{x  y}  =  [\prod_{p=1}^q U^{\veps_{q}}_i]_{x_p y_p},$$ with $ x = (x_1, \ldots, x_q) \in \INT{n}^q$ and 
$y = (y_1, \ldots, y_q)\in \INT{n}^q$.  We also set 
$$
\vec E = \INT{n}^q \times \INT{2d}.
$$
In analogy with usual non-bactracking matrices, an element $e = (x,i)$ of $\vec E$ is thought as the directed edge attached to $x$ 
associated to the $i$-th unitary matrix $V_i$.   If $e = (x,i), f = (y,j) \in \vec E$, the matrix-valued entry $(e,f)$ of $B$ defined 
in \eqref{eq:defBV2} is 
\begin{equation}\label{eq:defBV20}
B_{ef} = a_i [V_i]_{xy} \IND_{i \ne j^*} \in M_r(\dC) .
\end{equation}

We start with the Weyl formula for the spectral radius. If $o$ is the unit of $F_{d}$, we set 
$$
\rho_k =  \PAR{ \max \|    B_\star^k   \varphi\otimes \delta_e \|  }^{\frac {1 }{k }},
$$
where the maximum is over all $e  = (o,i), i \in \INT{2d}$ and $\varphi \in \dC^r$ of unit Euclidean norm. From the Weyl formula, $\rho_k$ converges 
to $\rho( B_\star)$ as $k$ goes to infinity (see for example \cite[Theorem 1.3.6]{MR1074574}). In the remainder of the proof, we fix 
$\veps >0$, then there exists an integer $k_0$ large enough such that for all $k \geq k_0$,
\begin{equation}\label{eq:defrho}
\rho_k \leq \rho =  \rho (B_\star) + \veps.
\end{equation}

Let $\ell$ be an integer. We fix an integer $\theta$. To get a good probabilistic estimate, we will upper bound the 
expectation of $\NRM{ B ^ {\ell}}^{2\theta}$ for $\theta$ large enough. We write
$$
\NRM{ B ^ {\ell}}^{2\theta} = \NRM{ B ^ {\ell}  ( B ^ {\ell })^* }^{ \theta } = \NRM{ \PAR{ B ^ {\ell} (B ^ {\ell })^* }^{\theta}}.
$$
In particular, if $\tr_{\dC^r}$ denotes the 
partial
trace on $\dC^r$ for operators on $\dC^r \otimes \dC^{\vec E}$, we get
$$
\NRM{ B ^ {\ell}}^{2\theta}\leq \tr \BRA{  \tr_{\dC^r} \BRA{ \PAR{ B ^ {\ell} (B ^ {\ell })^* }^{\theta}  }}.
$$
We expand the trace in terms of the matrix-valued entries of $B$:
\begin{eqnarray*}
\tr_{\dC^r} \BRA{ \PAR{ B ^ {\ell} (B ^ {\ell })^* }^{\theta}  } &=&  \sum_{e_\alpha \in \vec E , \alpha \in \INT{2\theta}}  \prod_{\alpha=1}^{\theta} (B ^ { \ell} )_{e_{2\alpha-1} e_{2\alpha}} ((B ^ { \ell} )^ *)_{e_{2\alpha}e_{2\alpha+1}},
\end{eqnarray*}
with $e_{2\theta +1} =e_1$. We expand further and use that $(B^*)_{ef} = (B_{fe})^*$, we obtain
\begin{eqnarray*}
\tr_{\dC^r} \BRA{ \PAR{ B ^ {\ell} (B ^ {\ell })^* }^{\theta}  }  & =& \sum_{\gamma} \prod_{\alpha=1}^{2 \theta}   \prod_{t=1}^\ell B^{\veps_\alpha}_{\gamma^\alpha_{t} \gamma^\alpha_{t+1}} ,
\end{eqnarray*}
with $B_{ef}^{\veps_\alpha}$ is equal to $B_{ef}$ for odd $\alpha$ and $(B_{fe})^*$ for even $\alpha$ and the  sum is over all 
$\gamma = (\gamma^1,\ldots,\gamma^{2\theta})$, $\gamma^\alpha  = (\gamma^\alpha_1,\cdots,\gamma^\alpha_{\ell+1})$  in 
$\vec E^{\ell+1}$ with the boundary conditions, for all $\alpha \in \INT{\theta}$, \begin{equation}\label{eq:bdcond}
\gamma^{2\alpha}_1 = \gamma^{2\alpha+1}_1 \AND \gamma^{2\alpha}_{\ell+1} =  \gamma^{2\alpha-1}_{\ell+1},
\end{equation}
with the convention $\gamma^{2\theta +1} = \gamma^1$.%

 Let us write
 $\gamma^\alpha_t = (x^\alpha_t,i^\alpha_t)$, $x^\alpha_t \in \INT{n}^q$ and $i^\alpha_t \in \INT{2d}$. From \eqref{eq:defBV20}, we find
$$
\tr_{\dC^r} \BRA{ B ^ { \ell} (B ^ { \ell} )^* } =   \sum _{\gamma \in P_{\ell,\theta}} \prod_{\alpha=1}^{2\theta} \prod_{t=1}^\ell a_{i^\alpha_t}^{\veps_\alpha} [V_{i^\alpha_t}]^{\veps_\alpha}_{x^\alpha_t x^\alpha_{t+1}}    ,
$$
where $(a^{\veps_\alpha}_i,[V_{i}]_{xy}^{\veps_\alpha})$ is equal to $(a_i, [V_i]_{xy})$ or $(a_i^* ,[\bar V_i]_{xy})$ depending on the 
parity of $\alpha$ and $P_{\ell,\theta}$ is the set of all $\gamma = (\gamma^1,\ldots ,\gamma^{2\theta})$ as above which are also 
non-backtracking, that is, for all $t \in \INT{\ell}$ and $\alpha \in \INT{2\theta}$,
\begin{equation}\label{eq:NBi}
i^\alpha_{t+1} \ne {i^{\alpha}_t}^*.
\end{equation}
(Note that $[\bar V_i]_{xy} = [V_{i^*}]_{yx}$ by construction).
We finally take the expectation and find:
\begin{equation}\label{eq:tr1}
\dE \NRM{ B ^ {\ell}}^{2\theta} \leq r  \sum_{\gamma \in P_{\ell,\theta}}   a(\gamma) p(\gamma) ,
\end{equation}
where we have introduced the algebraic and probabilist weights of an element $\gamma  \in P_\ell$ defined as
$$
a(\gamma) =  \prod_{\alpha = 1}^{2\theta} \NRM{ \prod_{t=1}^\ell a^{\veps_\alpha}_{i^\alpha_t} }  \quad   \hbox{ and } \quad p(\gamma) =  \left|  \dE \BRA{ \prod_{\alpha = 1}^{2 \theta}\prod_{t=1}^\ell   [V_{i^\alpha_t}]^{\veps_\alpha}_{x^\alpha_t x^\alpha_{t+1}}} \right|.
$$

We organize the terms on the right-hand side of \eqref{eq:tr1} by introducing the following equivalence class on elements of 
$P_{\ell,\theta}$. 
We write $\gamma \sim \gamma'$ if there exist a permutation $\tau$ on $\INT{n}$ and a family of permutations $(\beta_x)_{x \in \INT{n}}$ on $\INT{2d}$ such that the image of $\gamma = (x^\alpha_{t,p},i^\alpha_{t,p})$ by these permutations is $\gamma' = ((x^\alpha_{t,p})',(i^\alpha_{t,p})')$. More precisely, if for all $\alpha \in \INT{2\theta}$, $t \in \INT{\ell}$ and $p \in \INT{q}$, $\beta_{x^\alpha_{t,p}}(i_t^\alpha) = (i_t^\alpha)' = \beta_{x^\alpha_{t+1,p}} ((i_t^\alpha )^*)^*$ and $\tau(x^\alpha_{t,p}) = (x^\alpha_{t,p})'$.  
In more combinatorial language, if 
$P_{\ell,\theta}$ is seen as a collection of labeled paths where the labels are the vertex entries in $\INT{n}$ and the edge colors in $\INT{2d}$; an equivalence class is the corresponding unlabeled path.
We will use colored graphs defined formally as follows. 

\begin{definition}[Colored edge and graph]\label{def:colorgraph}
Let $X$ be a countable set and $C$ a countable set with an involution $* : C \to C$. A {\em colored edge} $[x,i,y]  $ is an equivalence 
class of $X \times C \times X$ equipped with the equivalence relation $(x,i,y) \sim (y,i^*,x)$. A {\em colored graph} $G = (V,E)$ is the 
pair formed by a vertex set $V \subset X$ and a set of colored edges $[x,i,y]$ with $x,y \in V$. The {\em degree} of $v \in V$ is
 $\sum_{e = [x,i,y] \in E}  ( \IND ( v = x )  + \IND ( v = y) )$.
\end{definition}

We now define a colored graph naturally associated with an element $\gamma \in P_{\ell,\theta}$. We write 
$x^{\alpha}_t = (x^{\alpha}_{t,1}, \ldots,x^{\alpha}_{t,q}) \in \INT{n}^q$  and define the colored graph 
$G_\gamma = (V_\gamma,E_\gamma)$  with color set $\INT{2d} $, 
$V_\gamma = \{ x^\alpha_{t,p} : \alpha \in \INT{2\theta}, t \in \INT{\ell+1},  p \in \INT{q}\} \subset [n]$ and 
$E_\gamma = \{ [ x^\alpha_{t,p} ,i^\alpha_t, x^\alpha_{t+1,p}] :  \alpha \in \INT{2\theta}, t \in \INT{\ell} , p \in \INT{q} \}$. 
The {\em multiplicity} of an edge $e \in E_\gamma$ is defined as 
$m(e) = \sum_{t\in \INT{\ell},\alpha \in \INT{2\theta},p\in \INT{q}} \IND (  [ x^\alpha_{t,p} ,i^\alpha_t, x^\alpha_{t+1,p}] = e )$. 
If  $e = |E_\gamma|$, $e_1$ is the number of edges of multiplicity one and $e_{\geq 2}$ is the number of edges of multiplicity 
at least two, we have 
$$
e = e_1 + e_{\geq 2} \AND e_1 + 2 e_{\geq 2} \leq 2 q \ell \theta. 
$$
We deduce that
\begin{equation}\label{eq:bdee1}
e \leq q \ell\theta + e_1 /2. 
\end{equation}
Moreover, for each $p \in \INT{q}$ and $\alpha \in \INT{2\theta}$, the sequence $\gamma^\alpha_p = ((x^\alpha_{t,p},i^\alpha_{t,p},x^\alpha_{t+1,p}))_{t \in \INT{\ell}}$ is a path in $G_\gamma$. Besides, from the boundary condition \eqref{eq:bdcond}, for fixed $p \in \INT{q}$, the sequence of paths $(\gamma^\alpha_p), \alpha \in \INT{2\theta},$ are connected by their ends. It follows that the graph $G_\gamma$ has at most $q$ connected components. 
If $v =|V_\gamma|$, we deduce that
\begin{equation}\label{eq:genus0}
e - v + q \geq 0.
\end{equation}
In particular, combining the last two inequalities, for any $\gamma \in P_{\ell,\theta}$, 
\begin{equation}\label{eq:genus}
\chi = q (\ell \theta+1) + e_1/2 - v \geq 0.
\end{equation}

For integers $v, e_1$ such that \eqref{eq:genus} holds,  we denote by $P_{\ell,\theta} (v,e_1)$ the set of elements in $P_{\ell,\theta}$ 
with $v$ vertices and $e_1$ edges of multiplicity one.  Note that if $\gamma \sim \gamma'$, the number of vertices and the number of 
edges with a given multiplicity are equal. It follows that we may define $\cP_{\ell,\theta} (v,e_1)$ as the set of equivalence classes 
with $v$ vertices and $e_1$ edges of multiplicity one.  Our first lemma is a rough bound on $\cP_{\ell,\theta} (v,e_1)$. 

\begin{lemma}\label{le:Plclass}
If $\chi  =  q (\ell \theta+1) + e_1/2 - v $, we have, 
$$
|\cP_{\ell,\theta} (v,e_1) |\leq v^{q-1} ( 2 q \ell \theta) ^{ 6 \chi }.  
$$
\end{lemma}
\begin{proof}
We order the set $T = \INT{q} \times \INT{2\theta} \times \INT{\ell+1} $ with the lexicographic order on the first two coordinates and 
for the last coordinate: $(p,\alpha,t) \succ (p,\alpha,t')$ if $\alpha$ is odd and $t > t'$ or if $\alpha$ is even and $t < t'$. In words, 
the natural order is reversed for the last coordinate depending on the parity of $\alpha$.

We think of an element $s= (p,\alpha,t) \in T$ as a time. We also define the set $\vec E_1 = \INT{n} \times \INT{2d}$ ordered with the 
lexicographic order.  An element $\gamma \in P_{\ell,\theta}$ can be written as the sequence 
$ (\gamma_{s})_{s \in T} \in \vec E_1^T$ with $\gamma_s = (y_s, i_s) \in \vec E_1$.  For concreteness, we may define a canonical element 
in each equivalence class as follows. We say that $\gamma \in P_{\ell,\theta}$ is {\em canonical} if $(\gamma_s)_{s \in T}$ is minimal 
for the lexicographic order in its equivalence class. For example, if $\gamma \in P_{\ell,\theta} (v,e_1)$ is canonical then 
$\gamma_{(1,1,1)} = (1,1)$, $V_\gamma = \INT{v}$ and the vertices appear for the first time in the sequence $(\gamma_s)$ in order.  
Our goal is then to give an upper bound on the number of canonical elements in $P_{\ell,\theta} (v,e_1)$.

We define $T_0\subset T$ as the set of $(p,\alpha,t)$ such that if $\alpha$ is odd $t \in \INT{\ell}$ and if $\alpha$ is even $t-1 \in \INT{\ell}$. 
For $s \in T$, we write $s+1$ for the successor of $T$ in the total order and $s-1$ for its predecessor, with the convention that $(1,1,1) - 1= 0$ 
and $(q,2\theta,1)+1 = \infty$.  The set $T_0$ is the subset of $s$ in $T$ such that $s$ and $s+1$ have the same first two coordinates. 
We explore iteratively the sequence  $(i_s,y_{s+1})$, $s \in T_0$, and we also build a growing sequence of forests $(F_s)$ (graphs without cycles). 
We define $F_0$ as the graph with no edge and vertex set $\{1 \}$. At time $s \in T_0$, if the addition to $F_{s-1}$ of the edge 
$[y_s,i_s,y_{s+1}]$ does not create a cycle, then $F_s$ is the forest spanned by $F_s$ and $[y_s,i_s,y_{s+1}]$, we then say that 
$[y_s,i_s,y_{s+1}]$ is a {\em tree edge}. Otherwise $F_s = F_{s-1}$.

We now build a partition of $T_0$. Let us say that a time $s \in T_0$ is a {\em first time} if the vertex 
$y_{s+1}$ has not been seen before.  Then,  $[y_s,i_s,y_{s+1}]$ is a tree edge that is said to be associated with $y_{s+1}$.  We call the 
vertices $y_{(p,1,1)} \in V_{\gamma}$ with $p \in \INT{q}$, the {\em seeds}.  Due to the boundary condition \eqref{eq:bdcond}, each vertex 
$y \in V_\gamma$ different from a seed has an associated tree edge. Hence, the number of tree edges, say $f$, satisfies the inequality
$$
f  \geq v - q.
$$ 
If $s$ is a first time, then the value of $y_{s+1}$ is determined by the preceding values $y_{s'+1}$, $s' < s$ and the value of the 
seeds: we have $y_{s+1} = u+1$ where $u$ is the number of distinct vertices which had been seen so far. We say that a time 
$s \in T_0$ is a {\em tree time} if  $[y_s,i_s,y_{s+1}]$ is a tree edge that has already been visited. Finally, the other times $s \in T_0$ are 
called {\em important times}. We have thus partitioned $T_0$ into first times, tree times, and important times.

Due to the non-backtracking constraint \eqref{eq:NBi}, a first time cannot be directly followed by a tree time. The sequence $\gamma$ can thus be decomposed as 
the successive repetitions of: $(i)$ a sequence of first times (possibly empty), $(ii)$ an important time or a time in $T \backslash T_0$, 
$(iii)$ a sequence of tree times (possibly empty).

We note that all edges are visited at least twice except $e_1$ of them. We deduce that the number of important times is at most 
$$
2 q \ell \theta - 2 f + e_1 \leq 2 q (\ell\theta +1) - 2 v + e_1 = 2 \chi. 
$$

We mark the important times by the vector $(i_s,y_{s+1},y_{\tau},i_\tau)$ where $\tau > s$ is the next time which is not a tree time. 
We claim that the canonical sequence $\gamma$ is 
uniquely determined by the value $y_{(p,1,1)}$ of the seeds, the positions of the important times, and their marks. Indeed, the sequence $(y_{s+1},i_{s+1}, \ldots, y_{\tau})$ is the unique non-backtracking path in $F_s$ from $y_{s+1}$ to $y_\tau$ (there is a unique non-backtracking path between two vertices of a tree). Moreover, if $\sigma \geq \tau$ is the next important time or time in $T \backslash T_0$, $(\tau, \ldots, \sigma-1)$ is a sequence of first times (if $\tau = \sigma$ this sequence is empty). It follows that if $u$ vertices had been seen by time $s$, we have $y_{\tau + k} = u+k$, for $k = 1, \cdots, \sigma- \tau$, by the minimality of canonical paths. Similarly, if $i_\tau \ne 1^* = d+1$, then $i_{\tau + k} = 1$ for  $k = 1, \cdots, \sigma- \tau$, while if $i_\tau = d+1$, then $i_{\tau + 1} = 2$, $i_{\tau + k} = 1$ for $k = 2, \cdots, \sigma- \tau$.

There are at most $2q\ell \theta$ possibilities for the position of an important time and $(2q\ell\theta )^2$ possibilities for its mark 
(there are at most $2 q \ell \theta$ edges in $G_\gamma$). In particular, the number of distinct canonical paths in $P_{\ell,\theta} (v,e_1)$ 
is at most 
$$
v^{q-1} (2 q \ell \theta )^{6 \chi},
$$
where $v^{q-1}$ accounts for the possible values of the seeds (recall that $y_{(1,1,1)} = 1$).  \end{proof}

We now estimate the probabilistic weight $p(\gamma)$. To this end, we introduce a finer partition of the set $P_{\ell,\theta}$. 
Let $\gamma  \in P_{\ell,\theta}$ with $\gamma^{\alpha}_t = (x^{\alpha}_t,i^{\alpha}_t), t \in \INT{\ell+1}$. We define a {\em bracket} as a 
colored edge $[x, i ,y]$ on $\INT{n}^q$ with color set $\INT{2d}$  (in the sense of Definition \ref{def:colorgraph}). We associate to 
$\gamma$ the sequence of brackets $([x^{\alpha}_t, i^{\alpha}_t ,x^{\alpha}_{t+1}])$, $t \in \INT{\ell}, \alpha \in \INT{2\theta}$. We say that 
a bracket $[x^{\alpha}_t, i^\alpha_t ,x^{\alpha}_{t+1}]$ of $\gamma$ is {\em isolated} if for all $p \in \INT{q}$, the colored edge of 
$G_\gamma$, $e = [x^{\alpha}_{t,p}, i^{\alpha}_t ,x^{\alpha}_{t+1,p}]$ is not visited at a different time: that is for all  
$(t',\alpha',p') \in \INT{\ell}\times \INT{2\theta}\times \INT{q}$ with $(t',\alpha') \ne (t,\alpha)$, we have 
$e \ne [x^{\alpha'}_{t',p'}, i^{\alpha'}_{t'} ,x^{\alpha'}_{t'+1,p'}]$.  For integer $b$, we denote by  $P_{\ell,\theta} ( v,e_1,b)$ the set of 
$\gamma \in P_{\ell,\theta} (v,e_1)$ with $b$ isolated brackets. 

\begin{lemma}\label{le:pgamma}
Let $\gamma \in P_{\ell,\theta} (v, e_1,b)$ and $\chi = q (\ell\theta+1) + e_1/2 - v$. If $v > q(\ell\theta+1)$ then $p(\gamma ) = 0$. Otherwise, 
there exists a universal constant $c >0$ such that 
$$
p(\gamma) \leq c n^{ -  q \ell \theta}  \eta^{b + \frac{e_1}{q}} \left( 2  q\ell \theta \right) ^{ 2\chi  },
$$
with $\eta = (c q \ell\theta)^{q/2} n^{-1/8}$.
\end{lemma}

\begin{proof}
Let $T, T_0$ be equipped with their total order as in the proof of Lemma \ref{le:Plclass}. By Theorem \ref{theorem-with-brackets}, 
$p(\gamma) = 0$ unless the sequence $(x^{\alpha}_{t,p},x^{\alpha}_{t+1,p}), (p,\alpha,t) \in T_0$ is an even sequence. In particular, this 
last condition and \eqref{eq:NBi} imply that $2 ( v - q) \leq 2 q \ell\theta $, since for all vertices $x$ of $V_\gamma$ different from the seeds 
$(x_{(p,1,1)}), p \in \INT{q},$ we may associate at least two elements $(p,\alpha,t) \in T_0$ such that $ x= x^{\alpha}_{t+1,p}$ (for $\alpha$ odd) 
or $ x= x^{\alpha}_{t,p}$ (for $\alpha$ even). It gives the first claim.

We now prove the second claim. By Corollary \ref{cor:WG2}, applied to $k = 2 q \ell \theta$, it suffices to prove that 
\begin{equation}\label{eq:m3bd}
m_{\geq 4} \leq  4 \chi ,
\end{equation}
where $m_{\geq 4} = \sum_{e} \IND(m_e \geq 4) m_e$ is the sum of multiplicities of edges of $G_\gamma$ with multiplicity at least $4$. 
Indeed, let $e_{23}$ be the number of edges of multiplicity $2$ or $3$. We have 
$$
e_1 + 2 e_{23} + m_{\geq 4} \leq 2 q \ell \theta \AND e_1 + e_{23} + \frac{m_{\geq 4}}{ 4} \geq e \geq v - q,
$$
where $e = |E_\gamma|$ is the number of edges and where we have used \eqref{eq:genus0}. 
We cancel $e_{23}$, and find
$$
-e_1 + \frac{m_{\geq 4}} { 2} \leq 2 q \ell \theta - 2 v + 2 q. 
$$
We obtain \eqref{eq:m3bd}. \end{proof}

Our final lemma is the estimation of $\sum a(\gamma)$ where the sum is over an equivalence class. This is our main combinatorial ingredient.

\begin{lemma}\label{le:agamma}
Let $\gamma \in P_{\ell,\theta} (v, e_1,b)$, $\chi = q (\ell\theta +1) + e_1/2 - v$ and let $\rho$ and $k_0$ be as in \eqref{eq:defrho}. There exists a constant 
$c >1$ (depending on $k_0,r,d$ and $\max_i \| a_i \|$), such that 
$$
A(\gamma) = \sum_{\gamma' : \gamma'\sim \gamma} a(\gamma') \leq n^v c^{b + q\theta + \chi} \rho^{2 \ell\theta}.
$$
\end{lemma}

We start with a preliminary lemma. 
\begin{lemma}\label{le:agamma0}
Let $\rho$  and $k_0$ be as in \eqref{eq:defrho}. There exists a constant $c > 1$ (depending on $k_0,d$ and $\max_i \| a_i \|$) such that for any integer 
$k \geq 1$, 
$$
\sum \NRM{\prod_{t=1}^k a_{i_t} }^2 \leq   c^2 \rho^{2k}
\AND 
\sum \NRM{\prod_{t=1}^k a_{i_t} } \leq  c^k  \rho^{k},
$$
where the sums are over all $(i_1, \cdots, i_k)$ such that $i_{t+1} \ne i_t^*$.  In particular, 
$$\max \NRM{\prod_{t=1}^k a_{i_t} } \leq c\rho^k$$ 
where the maximum is over all $(i_1, \cdots, i_k)$ such that $i_{t+1} \ne i_t^*$. 
\end{lemma}

\begin{proof}
Let $k \geq 1$ be an integer. As above Proposition \ref{prop:nonback2}, for 
$e,f \in \Fd \times \INT{2d}$, we denote by $(B^k_\star)_{ef} \in M_r(\dC)$ the matrix $P_e B^k _\star P_f$ with $P_f$ the orthogonal projection onto 
$\dC^r \otimes \delta_f$.  
Let $e = (o,i)$ and $f = (g,j) $ be in $\Fd \times \INT{2d}$ where $g = g_{i_1}\cdots g_{i_m}$ is written in reduced form (that is, for all $t\in \INT{m-1}$, $i_{t+1} \ne i_t^*$). If $m = k$,  $i_1 = i$ and $j \ne i^*_{m}$, from the definition of $B_\star$, we  have 
$$
(B_\star^k)_{e f} =  \prod_{t=1}^k a_{i_t}.
$$
Otherwise, $m \ne k$, $i_1 \ne i$ or $j = i^*_{m}$ and we find $(B_\star^k)_{e f}  = 0$. For $\varphi \in \dC^r$, we deduce that
$$
\| B_\star ^k \varphi \otimes \delta_e \|^2 = (2 d - 1) \sum \| \prod_{t=1}^k a_{i_t} \varphi \|^2, 
$$
where the sum  is over all $(i_1, \cdots, i_k)$ such that $i_1 = i$ and $i_{t+1} \ne i_t^*$.  Recall that $\rho_{k}$ was defined above \eqref{eq:defrho}. Taking the supremum over all $\varphi$ of unit norm, we get
$$
\rho_{k}^{2k} = (2d-1) \cdot \max_{i \in \INT{2d}, \varphi \in S^{r-1}} \sum \NRM{\prod_{t=1}^k a_{i_t} \varphi }^2 ,
$$
where the sum  is over all $(i_1, \cdots, i_k)$ such that $i_1 = i$ and $i_{t+1} \ne i_t^*$. Observe that if $T \in M_r(\dC)$ is a matrix with singular values $\| T \| = s_1 \geq \ldots \geq s_r \geq 0$ and right singular vectors $u_1, \ldots , u_r$, then 
$$
\| T \varphi \|^2 = \sum_{i=1}^r s_i ^2 | \langle u_i , \varphi \rangle |^2. 
$$
Hence, if $\varphi$ is uniformly distributed on $S^{r-1}$, we find 
$$
\dE \| T \varphi \|^2 = \frac 1 r \sum_{i=1}^r s_i ^2 \geq \frac{\| T \|^2 }{r} . 
$$
In particular, from what precedes, 
$$
\rho_{k}^{2k} \geq \frac{(2d-1)}{r} \max_{i \in \INT{2d}} \sum \NRM{\prod_{t=1}^k a_{i_t} }^2 .
$$
We find from \eqref{eq:defrho} that for all $k \geq k_0$,
$$
\sum \NRM{\prod_{t=1}^k a_{i_t} }^2 \leq  \frac{2d r }{2d-1} \rho^{2k},
$$ 
where the sum  is over all $(i_1, \cdots, i_k)$ such that and $i_{t+1} \ne i_t^*$.
Up to a large multiplicative constant on the right-hand side (depending on $\max_i \|a_i\|$, $k_0, r, d$), this last inequality is valid for all $k \geq 1$. This is the first claimed statement. 
For the second statement, we apply Cauchy-Schwarz inequality and get:
$$
\sum \NRM{\prod_{t=1}^k a_{i_t} } \leq \PAR{2d (2d-1)^{k-1}}^{\frac 1 2}   \PAR{\sum \NRM{\prod_{t=1}^k a_{i_t} }^2 }^{\frac 1 2}. 
$$
Modifying the constant $c$, we obtain the second statement. The last statement is an immediate consequence of the first statement.
\end{proof}

We may now prove Lemma \ref{le:agamma}. 
\begin{proof}[Proof of Lemma \ref{le:agamma}] We start by introducing a new decomposition of a path $\gamma \in P_\gamma$. This decomposition is tailored to use the three bounds given by Lemma \ref{le:agamma0} but keeping in mind that for each use of Lemma \ref{le:agamma0}, there is a cost of a constant factor $c$ for the first and third bounds and an exponential cost $c^k$ for the second bound. The conclusion of Lemma \ref{le:agamma} will ultimately follow from an application of the first and last inequality $O(\chi + q \theta)$ times and the second $O(b)$ times.

As usual,  we write $\gamma \in P_{\ell,\theta}$  with $\gamma^{\alpha}_t = (x^{\alpha}_t, i^{\alpha}_t) \in \INT{n}^q \times \INT{2d}$, 
$x^{\alpha}_t = (x^{\alpha}_{t,p})_{p \in \INT{q}}$. Let $T = \INT{2\theta} \times \INT{\ell}$ and let $T_{\geq 3}$ be the set of $ (\alpha,t) \in T$ 
such that for some $p \in \INT{q}$, $x^\alpha_{t,p}$  has a degree at least $3$ in $G_\gamma$ (recall that the degree of a vertex was defined in 
Definition \ref{def:colorgraph}). We let $T_1$ be the set of $(\alpha,t) \in T$ such that $x^\alpha_{t,p}  \in  \{ x^{\alpha'}_{t',p'} :  \alpha'\in \INT{2\theta}, t' \in \{1,\ell+1\}, p' \in \INT{q} \}$. In words, $T_1$ is the set of indices $(\alpha,t)$ such that $x^\alpha_{t,p}$ reaches a boundary vertex for some $p \in \INT{q}$. We note that, in particular, $(\alpha,1) \in T_1$. 

We finally set $T_\star = T _{\geq 3} \cup T_1 $.  We write $l  = |T_\star |$, 
$T_\star = \{ (\alpha,t_j^\alpha) : j \in \INT{l_\alpha} \}$ with $\sum_\alpha l_\alpha=l$, $(t_j^\alpha)_j$ increasing with $t^\alpha_1 = 1$ and by convention we set  $t^\alpha_{l_\alpha+1} = \ell+1$.We write $a(\gamma)$ 
defined below \eqref{eq:tr1} as
\begin{equation}\label{eq:ag10}
 a(\gamma)  =  \prod_{\alpha =1}^{2\theta} \NRM{  \prod_{j = 1}^{l_{\alpha} }  \prod_{t=t^\alpha_j}^{t^ \alpha_{j+1} -1 } a^{\veps_\alpha}_{ {i^{\alpha}_{t}} } }.
\end{equation}

We start with an upper bound on $|T_{\star}| $. Similarly to \eqref{eq:m3bd}, let $m_{\geq 3} = \sum_e \IND(m_e \geq 3) m_e$ be the sum of multiplicities at least $3$, i.e. $m_{ \geq 3}$ is the number of indices $(\alpha,t,p) \in \INT{2\theta} \times \INT{\ell} \times \INT{q}$ where the path visits an edge visited at least $3$ times. Let $v_k$ be the number of 
vertices of $V_\gamma$ with degree $k$ and $d_{\ge 3} = \sum_{k \geq 3} k v_k$.  By definition, $d_{\geq 3}$ is the number of edges neighboring a vertex of degree at least $3$.  Observe that if $(\alpha,t) \in T_{\geq 3}$, then there exists $p \in \INT{q}$ such that $x^{\alpha}_{t,p} $ has degree at least $3$ and the edge $[x^{\alpha}_{t,p} ,i^\alpha_t , x^{\alpha}_{t+1,p}]$ either is visited for the first or second time or is visited for the third time or more. Similarly, if $(\alpha,t) \in T_1 \backslash T_{\geq 3}$, then there exists $p \in \INT{q}$ such that $x^{\alpha}_{t,p} $ is a boundary vertex of degree  $1$ or $2$ and the edge $[x^{\alpha}_{t,p} ,i^\alpha_t , x^{\alpha}_{t+1,p}]$ either is visited for the first or second time or is visited for the third time or more. Since there are at most $2 q \theta$ boundary vertices,
we get the inequality 
\begin{equation}\label{eq:T3T3}
|T_{\star}| = |T_{\geq 3} \cup T_1| \leq 2  d_{\geq 3} +  m_{\geq 3} + 4 \times 2 q \theta.
\end{equation}
Arguing as in  \eqref{eq:m3bd}, we claim that $m_{\geq 3} \leq 6 \chi$. Indeed, if $e_2$ is the number of edges of multiplicity $2$, we get
$$
e_1 + 2 e_2 + m_{\geq 3} = 2 q \ell \theta \AND e_1 + e_2 + \frac{m_{\geq 3}}{3} \geq e \geq v - q, 
$$
where the last inequality is \eqref{eq:genus0}. Cancelling $e_2$, we get 
$$
-e_1 + \frac{m_{\geq 3} }{3}  \leq 2q \ell \theta - 2e.  
$$
Hence $m_{\geq 3} / 3 \leq 2q \ell \theta - 2v + 2 q + e_1 = 2 \chi$, as claimed.

We may also similarly estimate $d_{\geq 3}$. We write
$$
v_1 + v_2 + \frac{d_{\geq 3}}{3} \geq \sum_k v_k =  v \AND v_1 + 2 v_2 +  d_{\geq 3} = \sum_{k} k v_k = 2 e,
$$
where $e = |E_\gamma|$ is the number of edges. Canceling $v_2$, we find 
\begin{equation}\label{eq:sge3}
\frac{d_{\geq 3} }{3} \leq 2 ( e - v) + v_1 \leq  2 (q \ell \theta   + e_1/ 2 - v)  + 2 q \theta = 2 \chi + 2 q(\theta-1),
\end{equation}
where, in the second inequality, we have used \eqref{eq:bdee1} and the bound $ v_1 \leq 2q \theta $ is a consequence of the assumption \eqref{eq:NBi}: only the vertices 
$\gamma^{2 \alpha}_{1,p} = \gamma^{2\alpha+1}_{1,p} $ and $\gamma^{2\alpha}_{\ell+1,p} = \gamma^{2\alpha +1}_{\ell+1,p}$ with 
$p \in \INT{q}$ can possibly be of degree $1$ in $G_\gamma$.  Therefore, we deduce from \eqref{eq:T3T3} that
\begin{equation}\label{eq:boundl}
l \leq 18 \chi +20 q\theta.
\end{equation}

We now proceed to the bound of the factors in \eqref{eq:ag10} when we sum over all $\gamma' \sim \gamma$. We perform this by building a partition of the factors in \eqref{eq:ag10} and then summing within each block of the partition over all $\gamma' \sim \gamma$. 
Consider the color set $\{-1,1\}$ equipped with the trivial involution $c^* = c$. We then define a colored graph, say $\Gamma$, on the vertex set 
$T = \INT{2\theta} \times \INT{\ell}$ and on the color set $C$, by placing an edge $[(\alpha,t),1,(\alpha',t')]$ between $(\alpha,t) \ne (\alpha',t')$ with 
color $1$ if there exists $(p,p')$ in $\INT{q}$ such that 
$(x^{\alpha}_{t,p},i^{\alpha}_{t},x^{\alpha}_{t+1,p}) = (x^{\alpha'}_{t',p'},i^{\alpha'}_{t'},x^{\alpha'}_{t'+1,p'})$ and the edge 
$[(\alpha,t),-1,(\alpha',t')]$ with color $-1$ if 
$(x^{\alpha}_{t,p},i^{\alpha}_{t},x^{\alpha}_{t+1,p}) = (x^{\alpha'}_{t'+1,p'},{i^{\alpha'}_{t'}}^*,x^{\alpha'}_{t',p'})$. 
In words, we place an edge between two elements of $T$ if they share a common edge in $G_\gamma$, the color encoding the orientation. In particular,  by definition:

{\em Claim 1. }  $(\alpha,t) \in T$ has degree $0$ in $\Gamma$ iff $[x^{\alpha}_{t},i^{\alpha}_{t},x^{\alpha}_{t+1}]$ is an isolated bracket.

Summing over all possibilities for $\gamma'$, we find from \eqref{eq:ag10} that 
\begin{equation}
\label{eq:ag101}
I = \sum_{\gamma' : \gamma'\sim \gamma} a(\gamma')  \leq (2d -1)^{\theta} n^v \sum \prod_{\alpha= 1}^{2\theta} \NRM{  \prod_{j = 1}^{l_{\alpha} }  \prod_{t=t^{\alpha}_j}^{t^\alpha_{j+1} -1 } a_{ {i^{\alpha}_{t}} } },
\end{equation}
where the sum is over all $(i^{\alpha}_t)_{(\alpha,t) \in T}$ in $\INT{2d}$ such that for all $(\alpha,t), (\alpha',t')$: 
\begin{equation}\label{eq:sumindex}
\hbox{ $i^{\alpha}_{t+1} \ne {i^{\alpha}_t}^*$, $i^{\alpha}_{t} = i^{\alpha'}_{t'}$ if $[(\alpha,t),1,(\alpha',t')] \in \Gamma$ and 
$i^{\alpha}_{t} =(i^{\alpha'}_{t'})^*$ if $[(\alpha,t),-1,(\alpha',t')] \in \Gamma$.}
\end{equation}
The factor $(2d-1)^\theta$ in \eqref{eq:ag101} comes from the choices of $i^\alpha_{\ell+1}$.

For ease of notation, for $r = (\alpha,t)$ and $s = (\alpha , t')$ with $t < t'$ in $\INT{\ell+1}$ and $\alpha \in \INT{2\theta}$, we will denote 
by $[r,s)$ the sequence $((\alpha,t),(\alpha,t+1), \cdots, (\alpha,t' - 1))$ and write $u \in [r,s)$ if $u = (\alpha,t'')$ with $t \leq t'' < t'$. 
Such sequence $[r,s)$ will be called an {\em interval} of {\em length} $t'-t$. If $[r_1,s_1)$ and $[r_2,s_2)$ are two intervals, 
$r_i = (\alpha_i , t_i)$, $s_i = (\alpha_i , t'_i)$ we say that $[r_1,s_1)$ and $[r'_2,s'_2)$ are {\em $1$-paired} if they have the same length 
and, for all $0\leq \delta < t'_i - t_i$, either: $[ (\alpha_1,t_1 + \delta) , 1 ,(\alpha_2,t_2+\delta)]$ is an edge of $\Gamma$ or 
$[ (\alpha_1,t_1 + \delta) , -1 ,(\alpha_2,t'_2-1-\delta)]$ is an edge of $\Gamma$, see Figure \ref{fig:1-1}. We say that there are {\em $2$-paired} if there exists a 
third interval which is $1$-paired to both of them. Iteratively, we define {\em $k$-paired} intervals. Finally, we say that two intervals 
are {\em paired} if they are $k$-paired for some $k \geq 1$. The key property is that if two intervals $[r_1,s_1)$ and $[r_2,s_2)$ are paired then, 
with the above notation, either 
\begin{equation}\label{eq:pairing}
\hbox{for all $0\leq \delta < t'_i - t_i$: } i^{\alpha_1}_{t_1 + \delta} = i^{\alpha_2}_{t_2 + \delta} \quad  \hbox{ or } \quad \hbox{for all $0\leq \delta < t'_i - t_i$: }  i^{\alpha_1}_{t_1 + \delta} = (i^{\alpha_2}_{t'_2 - \delta-1})^*.
\end{equation}

\begin{figure}
\begin{center}
\begin{tikzpicture}[xscale=0.8]
\draw[-,cyan,ultra thick] (0,0) -- (4,0) ;

\draw (0.5,0) node[above]{$a$} ; 

\draw (1.5,0) node[above]{$b$} ; 

\draw (2.5,0) node[above]{$c$} ; 

\draw (3.5,0) node[above]{$d$} ;

\draw (0,0) node[left] {$t_{1}$};
\draw (4,0) node[right] {$t'_{1}$};

\foreach \x in  {0, ...,4} 
\draw[shift={(\x,0)},color=black] (0pt,3pt) -- (0pt,-3pt) ;

\draw (0.5,-1) node[above]{$a$} ; 

\draw (1.5,-1) node[above]{$b$} ; 

\draw (2.5,-1) node[above]{$c$} ; 

\draw (3.5,-1) node[above]{$d$} ;

\draw (0,-1) node[left] {$t_{2}$};
\draw (4,-1) node[right] {$t'_{2}$};

\draw[-,cyan,ultra thick] (0,-1) -- (4,-1) ; 

\foreach \x in  {0, ...,4} 
\draw[shift={(\x,-1)},color=black] (0pt,3pt) -- (0pt,-3pt) ;

\draw (0.5,-2) node[above]{$d^*$} ; 

\draw (1.5,-2) node[above]{$c^*$} ; 

\draw (2.5,-2) node[above]{$b^*$} ; 

\draw (3.5,-2) node[above]{$a^*$} ;

\draw (0,-2) node[left] {$t_{3}$};
\draw (4,-2) node[right] {$t'_{3}$};

\draw[-,cyan,ultra thick] (0,-2) -- (4,-2) ; 

\foreach \x in  {0, ...,4} 
\draw[shift={(\x,-2)},color=black] (0pt,3pt) -- (0pt,-3pt) ;

\end{tikzpicture}
\caption{Paired intervals of length $4$. The letters on the lines represent the different brackets that are equal. The first two are paired with color $1$, and the third with color $-1$. \label{fig:1-1}}
\end{center}
\end{figure}
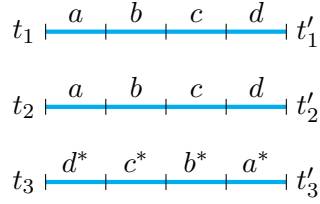

Our second claim asserts that the intervals $[t^{\alpha}_j,t^{\alpha}_{j+1})$ are paired unless they are made of isolated brackets.

{\em Claim 2. }  For any $(\alpha,j)$, the degree of  $(\alpha,t) \in T$ in $\Gamma$ is constant for all $t \in [t^{\alpha}_j,t^{\alpha}_{j+1})$, including the degree counted by colors. Moreover, if the degree is at least one, then the interval $ [t^{\alpha}_j,t^{\alpha}_{j+1})$ is paired to a subinterval of $ [t^{\alpha'}_{j'},t^{\alpha'}_{j'+1})$ for some $(\alpha',j') \ne (\alpha,j)$.

Let us check this claim. Assume that $t\in  [t^{\alpha}_j,t^{\alpha}_{j+1})$ and assume that $(\alpha,t)$ has a colored edges with color $1$ in $\Gamma$. Then, there exists $ (\alpha',t') \ne (\alpha,t)$ in $T$ and $(p,p') \in \INT{q}^2$ such that  $(x^{\alpha}_{t,p},i^{\alpha}_{t},x^{\alpha}_{t+1,p}) = (x^{\alpha'}_{t',p'},i^{\alpha'}_{t'},x^{\alpha'}_{t'+1,p'})$. If $ t+1 \in [t^{\alpha}_j,t^{\alpha}_{j+1})$ then $x^{\alpha}_{t+1,p} = x^{\alpha'}_{t'+1,p}$ is not a boundary vertex and has degree exactly $2$. Hence, $t+1 , t'+1 \le \ell$, $(\alpha',t'+1) \notin T_\star$ and from \eqref{eq:NBi}, the two adjacent edges of $x^{\alpha}_{t+1,p} = x^{\alpha'}_{t'+1,p}$ are $[x^{\alpha}_{t,p}, i^{\alpha}_{t,p}, x^{\alpha}_{t+1,p}] = [x^{\alpha'}_{t',p'}, i^{\alpha}_{t',p'}, x^{\alpha}_{t'+1,p'}]$ and $[x^{\alpha}_{t+1,p}, i^{\alpha}_{t+1,p}, x^{\alpha}_{t+2,p}] = [x^{\alpha'}_{t'+1,p'}, i^{\alpha}_{t'+1,p'}, x^{\alpha}_{t'+2,p'}]$ (the last equality follows from the constraint that the degree of $x^{\alpha}_{t+1,p}$ is exactly $2$). We thus have proved that if $ t+1 \in [t^{\alpha}_j,t^{\alpha}_{j+1})$ then $ (\alpha,t+1)$ shares a colored edge with color $1$ in $\Gamma$ with $(\alpha',t'+1)$ and that $((\alpha,t),(\alpha,t+1))$ is paired with $((\alpha',t'),(\alpha',t'+1))$. Similarly, if $ t- 1 \in   [t^{\alpha}_j,t^{\alpha}_{j+1})$ then $x^{\alpha}_{t,p} = x^{\alpha'}_{t',p}$ is not a boundary vertex and has degree exactly $2$. We deduce as above that $ (\alpha,t-1)$ shares a colored edge with color $1$ in $\Gamma$ with $ (\alpha',t'-1)$ and that $((\alpha,t-1),(\alpha,t))$ is paired with $((\alpha',s'-1),(\alpha',s'))$. The same argument applies to edges of color $-1$. By induction, this proves the claim.

In view of Claim 2, our goal is then to find paired subintervals of $[t^{\alpha}_j,t^{\alpha}_{j+1})$  of long length in \eqref{eq:ag101} and then use Lemma \ref{le:agamma0} 
for each of them. More precisely, we claim that there exists a partition of the set $T$ made of $m$ intervals $[r_i,r_{i+1}),  i \in \INT{m},$ with 
$m \leq 8 l$ and a classification of the intervals $[r_i,r_{i+1})$ into  $3$ types. An interval $[r_i,r_{i+1})$  is of type (1) iif all $u \in [r_i,r_{i+1})$ 
are of degree $0$ in $\Gamma$. There is a matching of intervals of type (2) such that two matched intervals of type (2) are paired together (by 
`matching', we mean a pair partition, we use the word matching to make a distinction between `matched intervals' and `paired intervals' defined 
above). Finally, any $(\alpha,t) \in T$ in an interval of type (3) is connected in $\Gamma$ to an element in an interval of type (2). 

Let us first assume the existence of such partition and conclude the proof of Lemma \ref{le:agamma}. For $c \in \{1,3\}$, we let $m_c $ be 
the number of intervals of type (c) and $k_{c,j}$ be the length of the $j$-th interval. Let $2m_2$ be the number of intervals of type (2) and 
$k_{2,j}$  be the common length of the two $j$-th paired intervals. We get from \eqref{eq:ag101} and \eqref{eq:pairing}
\begin{align*}
I  &\leq  (2d)^\theta n^v \prod_{j=1}^{m_1} \BRA{ \sum_{\substack{1 \leq t \leq k_{1,j} \\ i_{t+1} \ne i^*_t  }} \NRM{  \prod_{t = 1}^{k_{1,j}}  a_{ {i_{t}} } }}\prod_{j=1}^{m_2} \BRA{ \sum_{\substack{1 \leq t \leq k_{2,j} \\ i_{t+1} \ne i^*_t  }} \NRM{  \prod_{t = 1}^{k_{2,j}}  a_{ {i_{t}} } }^2}  \prod_{j=1}^{m_3} \BRA{\max_{\substack{1 \leq t \leq k_{3,j} \\ i_{t+1} \ne i^*_t  }}  \NRM{  \prod_{t = 1}^{k_{3,j}}  a_{ {i_{t}} } }  } \\
&= (2d)^\theta n^v \prod_{c=1}^3 \prod_{j=1}^{m_c} I_{c,j}.
\end{align*}
We note that to obtain the above expression, we have used Cauchy-Schwarz inequality for intervals of type (2) associated with the color $-1$: 
\begin{eqnarray*}
\sum_{\substack{1 \leq t \leq k \\ i_{t+1} \ne i^*_t  }} \NRM{  \prod_{t = 1}^{k}  a_{ {i_{t}} } }\NRM{  \prod_{t = 1}^{k}  a_{ {i^*_{k-t+1}} } } \leq \sqrt{\sum_{\substack{1 \leq t \leq k \\ i_{t+1} \ne i^*_t  }} \NRM{  \prod_{t = 1}^{k}  a_{ {i_{t}} } }^2 } \sqrt{ \sum_{\substack{1 \leq t \leq k \\ i_{t+1} \ne i^*_t  }}\NRM{  \prod_{t = 1}^{k}  a_{ {i^*_{k-t+1}} } }^2}  =  \sum_{\substack{1 \leq t \leq k \\ i_{t+1} \ne i^*_t  }} \NRM{  \prod_{t = 1}^{k}  a_{ {i_{t}} } }^2. 
\end{eqnarray*}
Also, to get the above bound on $I$,  we have used that for an interval of type (3), the $i^\alpha_t$'s associated with the indices $(\alpha,t)$ in the interval are determined by the  $i^\alpha_t$'s associated the indices $(\alpha,t)$ in intervals of type (2).

Then, by Lemma \ref{le:agamma0}, we have 
$I_{1,j} \leq c^{k_{1,j}}\rho^{k_{1,j}}$, $I_{2,j} \leq  c^{2}\rho^{2 k_{2,j}}$ and $I_{3,j} \leq  c \rho^{k_{3,j}}$. So finally, we deduce that 
$$
I  \leq (2d)^\theta n^v c^{b} c^{m}   \rho^{2 \ell\theta}.
$$
We are have used that $2 m_2 + m_3 \leq m$, $\sum_j k_{1,j} = b$ (by Claim 1) and $\sum k_{1,j} + 2 \sum k_{2,j} + \sum k_{3,j} = 2 \ell\theta $. Thus, 
from the assumption $m \leq 8  l$ and from \eqref{eq:boundl}, up to modifying the constant $c$, we deduce that the conclusion of 
Lemma \ref{le:agamma} holds provided the existence of the partition $[r_j,r_{j+1}), j \in \INT{m}$, of $T$.

The rest of the proof is devoted to the construction of the partition. To that end, we define a natural greedy pairing algorithm.  At a given step 
$i \geq 0 $ of the algorithm, there will be a partition of $T$ into intervals  $[r^{(i)}_j,r^{(i)}_{j+1})$.  Some intervals of the current partition will be 
said to be {\em active}, and the others are {\em frozen} (frozen intervals will remain unchanged in all subsequent partitions). The frozen intervals 
have been classified into types; the active intervals have not.  The algorithm stops when all intervals are frozen. At step $i=0$, the initial partition 
is $[t^{\alpha}_j,t^{\alpha}_{j+1})$, $\alpha \in \INT{2\theta}, j \in \INT{l_\alpha},$ and all intervals are active.  At each step of the algorithm, 
the number of active intervals decreases until there is none. Thus, the algorithm stops after at most $l$ steps.

We have the following induction hypothesis: each active interval $[r^{(i)}_j ,r^{(i)}_{j+1})$ is contained in  $[t^\alpha_{j'},t^{\alpha}_{j'+1})$ for some 
$\alpha \in \INT{2m}$ 
and $j' \in \INT{l_\alpha}$. Moreover, any $(\alpha,t)$ in a frozen interval of type (3) is connected in $\Gamma$ to an element in a frozen interval 
of type (2).   We now describe the induction step.  We pick the longest active sequence, say $[r^{(i)}_{j_0} ,r^{(i)}_{j_0+1})$. There are several 
cases to consider: 
\begin{enumerate}[(i)]
\item {\em All $u \in T$ with $u \in [r^{(i)}_{j_0} ,r^{(i)}_{j_0+1})$ have common degree $0$}. Then, we leave the partition unchanged, the interval 
$[r^{(i)}_{j_0} ,r^{(i)}_{j_0+1})$ is classified as type $(1)$ and becomes frozen. 
\item \label{lbii}{\em All $u \in T$ with $u \in [r^{(i)}_{j_0} ,r^{(i)}_{j_0+1})$ have common degree at least $1$, then we pick an interval $[s,s')$ 
such that $[s,s')$ and $[r^{(i)}_{j_0} ,r^{(i)}_{j_0+1})$ are $1$-paired (its existence is guaranteed by the induction hypothesis and Claim 2). The interval 
$[s,s')$ intersects a finite number, say $\alpha$, of intervals, $\cI = \{  [r^{(i)}_{j_1  + k - 1 },r^{(i)}_{j_1 + k}), k \in \INT{ \alpha} \}$.  We assume  
that the corresponding color is $1$.
} There are further subcases to consider:

\item[(ii-a)]{\em The intervals in $\cI$ are all active and $\cI$ does not contain $[r^{(i)}_{j_0} ,r^{(i)}_{j_0+1})$.}  Then, we merge the intervals in 
$\cI$ into $3$ intervals $[r^{(i)}_{j_1}, s)$, $[s,s')$ and $ [s',r^{(i)}_{j_1 + \alpha})$ (the two extreme intervals may be empty). The two extreme 
intervals are active. The intervals $[s,s')$ and $ [r^{(i)}_{j_0} ,r^{(i)}_{j_0+1})$ are frozen, they are classified as type (2), and they are matched 
together. The rest of the partition remains unchanged at step $i+1$. Note that if $\alpha = 1$, then the two extreme intervals are empty since we 
have picked the longest interval.

\item[(ii-b)]{\em The intervals in $\cI$ are all active and $\cI$ contains $[r^{(i)}_{j_0} ,r^{(i)}_{j_0+1})$.} We then have that $\alpha  \geq 2$ (the case 
$\alpha = 1$ is ruled out by the assumption that the corresponding color is $1$) and $[r^{(i)}_{j_0} ,r^{(i)}_{j_0+1})$ is one of the two extreme intervals 
of $\cI$. Assume, for example, that $j_1 = j_0$. We denote by $k$ and $k_0$ the lengths of $[r^{(i)}_{j_0} ,s )$ and $[r^{(i)}_{j_0} ,r^{(i)}_{j_0+1})$. 
Then, $1 \leq k \leq k_0$ and $h =k_0+k$ is the total length of $[r^{(i)}_{j_1} ,s')$. By construction if  $[(\alpha,t),(\alpha,t'))$ and 
$[(\alpha,t+  k),(\alpha,t'+ k))$ are contained $[r^{(i)}_{j_1} ,s')$ then they are paired.  By induction, we deduce that for any integer $a \geq 1$, if  $[(\alpha,t),(\alpha,t'))$ and 
$[(\alpha,t + a k),(\alpha,t'+ a k))$ are contained $[r^{(i)}_{j_1} ,s')$ then they are paired. 

We let $k'  = a k$ where $a \geq 1$ is the largest integer 
such that $a k \leq h/2$. We have $k'\geq h/4$ since otherwise, $k ' + k \leq 2 k' < h/2$. Assume, for example, that $h/4 \leq k' < h/3$. Let $r =h  - 3 k'$. 
We divide the sequence $(0,\ldots,h-1)$, into $(p k',\ldots,p k' + r-1)$ for $p \in \{0,1,2,3\}$ and $(qk' + r , \ldots, (q +1) k'-1)$, for $q \in \{0,1,2\}$. 
Mapping the interval $[r^{(i)}_{j_1} ,s')$ into the sequence $(0,\ldots, h-1)$, we have split  $[r^{(i)}_{j_1} ,r^{(i)}_{j_1+\alpha})$ into $7$ frozen intervals 
at step $i+1$ (see Figure \ref{fig:7} for an illustration). The two intervals corresponding to $p  \in \{0,1\}$ are of type $(2)$ and are matched together, similarly for the two intervals associated 
with $p \in \{2,3\}$. Finally, the three intervals corresponding to $q \in \{0,1,2\}$  are mutually paired; two of them are classified as type (2) and the 
remaining one as type (3).  The interval $[s',r^{(i)}_{j_1 + \alpha})$ becomes an active interval (if not empty). The rest of the partition remains 
unchanged at step $i+1$. The case  $h/3 \leq k' \leq  h/2$ is treated similarly with $5$ intervals, $3$ of length $r$ and $2$ of length $k' -r$ 
where $r = h - 2k'$ (if $r = 0$, $2$ intervals are sufficient), see Figure \ref{fig:7}.

\begin{figure}
\begin{center}
\begin{tikzpicture}[xscale=0.8]
\draw[-,cyan,ultra thick] (0,0) -- (20,0) ; 

\foreach \x in  {0,6,12,18} 
\draw (\x+0.5,0) node[above]{$a$} ; 

\foreach \x in  {0,6,12,18} 
\draw (\x+1.5,0) node[above]{$b$} ; 

\foreach \x in  {0,6,12} 
\draw (\x+2.5,0) node[above]{$c$} ; 

\foreach \x in  {0,6,12} 
\draw (\x+3.5,0) node[above]{$d$} ; 

\foreach \x in  {0,6,12} 
\draw (\x+4.5,0) node[above]{$e$} ; 

\foreach \x in  {0,6,12} 
\draw (\x+5.5,0) node[above]{$f$} ; 

\foreach \x in  {0, ...,20} 
\draw[shift={(\x,0)},color=black] (0pt,3pt) -- (0pt,-3pt) node[below] {$\x$};

\foreach \x in  {0, ...,20} 
\draw[shift={(\x,0.5)},color=black] (0pt,3pt) -- (0pt,-3pt) ;

\foreach \x in {0,6,12,18}
 \draw[shift={(\x,0)},-,color = black ,ultra thick] (-0.0,0.5) -- (1.0,0.5) ; 
 
\foreach \x in {0,6,12}
 \draw[shift={(\x,0)},-,color = black ,ultra thick] (2,0.5) -- (5.0,0.5) ; 

\draw (0,1) node {$r_{j_0}^{(i)}$};

\draw (6,1) node {$s$};

\draw (14,1) node {$r_{j_0+1}^{(i)}$};

\draw (20,1) node {$s'$};

\draw[-,cyan,ultra thick] (0,-3) -- (20,-3) ; 

\foreach \x in  {0,3,6,9,12,15,18} 
\draw (\x+0.5,-3) node[above]{$a$} ; 

\foreach \x in  {0,3,6,9,12,15,18} 
\draw (\x+1.5,-3) node[above]{$b$} ; 

\foreach \x in   {0,3,6,9,12,15} 
\draw (\x+2.5,-3) node[above]{$c$} ;

\foreach \x in  {0, ...,20} 
\draw[shift={(\x,-3)},color=black] (0pt,3pt) -- (0pt,-3pt) node[below] {$\x$};

\foreach \x in  {0, ...,20} 
\draw[shift={(\x,-2.5)},color=black] (0pt,3pt) -- (0pt,-3pt) ;

\foreach \x in {0,9,18}
 \draw[shift={(\x,-3)},-,color = black ,ultra thick] (-0.0,0.5) -- (1.0,0.5) ; 
 
\foreach \x in {0,9}
 \draw[shift={(\x,-3)},-,color = black ,ultra thick] (2,0.5) -- (8.0,0.5) ; 

\draw (0,-2) node {$r_{j_0}^{(i)}$};

\draw (3,-2) node {$s$};

\draw (17,-2) node {$r_{j_0+1}^{(i)}$};

\draw (20,-2) node {$s'$};

\end{tikzpicture}
\caption{Above : The splitting into $7$ intervals in the case $k=6$, $k_0 = 14$, $h = 20$, $k' = 6$, $r = 2$. The letters on the line represent the different brackets. Below : the splitting into $5$ intervals in the case $k=3$, $k_0 = 17$, $h=20$, $k' = 9$, $r=2$. \label{fig:7}}
\end{center}
\end{figure}
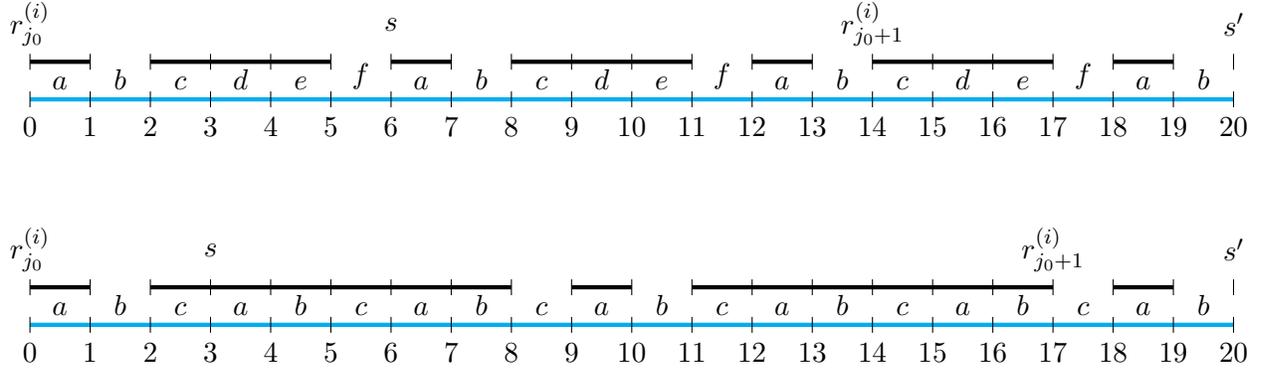

\item[(ii-c)]{\em The intervals in $\cI$ are all frozen.} Note that by the induction hypothesis, the intervals of $ \cI$ are of type (2) or (3), and all 
elements in $\cI$ are connected in $\Gamma$ to an element in a frozen interval of type (2). The interval $[r^{(i)}_{j_0} ,r^{(i)}_{j_0+1})$ becomes 
a frozen interval of type (3). The rest of the partition remains unchanged at step $i+1$. 

\item[(ii-d)]{\em The set $\cI$ contains both active and frozen intervals.} We split $\cI$ into at most three-set of contiguous intervals, say 
$\cI_1, \ldots, \cI_{k}$ with $k \in \{2,3\}$, such that $\cI_1$ contains only active intervals and the one or two others contain only frozen intervals. 
This is possible since, at each step, we have chosen the longest active interval: it follows that the union of frozen intervals is a union of disconnected 
intervals such that each of these disconnected intervals has a length larger or equal than the longest active interval.  We also split 
$[r^{(i)}_{j_0} ,r^{(i)}_{j_0+1})$ into two or three intervals, $J_1, \ldots, J_k$, such that for all $p \in \INT{k}$, $J_p$ is paired with $I_p$ which 
is contained in the union of the intervals in $\cI_p$. For each $p$, we may apply one of the above steps.

\item \label{lbiii}{\em As in \eqref{lbii} but the corresponding color is $-1$.} Up to minor modifications on the indices, the same splitting strategy \eqref{lbii} also works for 
color $-1$.

\end{enumerate}

At each step, the number of active intervals decreases. It follows that the algorithm stops after at most $l$ steps with a partition and a classification 
by types with all desired properties. Since $[r^{(i)}_{j_0} ,r^{(i)}_{j_0+1})$ is split into at most $7+1 = 8$ parts, the total number of intervals 
$m$ is bounded by $8l$. This concludes the proof of Lemma \ref{le:agamma}.
\end{proof}

All ingredients are gathered to prove Theorem \ref{th:FKB}. 

\begin{proof}[Proof of Theorem \ref{th:FKB}]
We fix $\veps >0$. For some small $\delta>0$ to be fixed later on, we set $$\theta = \lceil n^{\delta/q} \rceil.$$
We start by assuming that (for $n \geq 3$),
\begin{equation}\label{eq:choixql}
q \leq \frac{1}{2^7} \frac{\ln n}{\ln \ln n} \AND C_0 q \leq \ell \leq  2C_0 q  
\end{equation}
for some constant $C_0$ to be also fixed. In this proof, we write $c$ for a universal constant and $C$ for a constant which depends on $d$, 
$\veps$ and $\max_i \| a_i\|$. 
  
From \eqref{eq:tr1}, with $A(\gamma)$ as in Lemma \ref{le:agamma} and $\cP_{\ell,\theta} (v,e_1)$ as in Lemma \ref{le:Plclass}, we find 
\begin{equation*}
\dE \PAR{ \| B^{\ell} \|^{2\theta} } \leq  \sum_{v =1}^{q(\ell\theta +1)} \sum_{e_1 = 0}^\infty | \cP_{\ell,\theta} (v,e_1)| \max_{\gamma  \in \cP_{\ell,\theta}(v,e_1)}  p(\gamma)A(\gamma),
\end{equation*}
where we have used, by Lemma \ref{le:pgamma} that $p(\gamma) = 0$ unless $v \leq q (\ell\theta +1)$. Set $\chi = q (\ell \theta +1) +e_1/2 - v$. 
By Lemma \ref{le:pgamma} and Lemma \ref{le:agamma}, we have, 
\begin{equation}\label{eq:pa00}
 p(\gamma)A(\gamma) \leq  n^{ -  q \ell \theta}  \PAR{ (c q \ell \theta)^{q/2} n^{-1/8}}^{b + \frac{e_1}{q}} \left( c  q\ell \theta \right) ^{ 2\chi  } n^v C^{b + q\theta + \chi} \rho^{2 \ell\theta},
\end{equation}
where $b$ is the number of isolated brackets of $\gamma$.  For our choice of parameters in \eqref{eq:choixql}, for any constant $c >0$, we have 
for all $n$ large enough: $(cq)^q \leq n^{1/2^7}$. It follows that, if $\delta$ is small enough, for all $n$ large enough, 
$$
 (c q \ell \theta)^{q/2} n^{-1/8} \leq n^{1/2^7 + \delta/2 -1/8} \leq n^{-1/10}.
$$
We deduce that for all $n$ large enough, the factor in $b$ in \eqref{eq:pa00} is bounded by $1$. Therefore, for all $n$ large enough, we have
$$
 \max_{\gamma  \in \cP_{\ell,\theta}(v,e_1)}  p(\gamma)A(\gamma) \leq   n^{ -  q \ell \theta } n^{-\frac{e_1}{10q}}  \left( c  q \ell \theta  \right) ^{ 2\chi  } n^v C^{ q\theta + \chi} \rho^{2 \ell\theta}.
$$
We set $r = q (\ell \theta +1) -v$ and use Lemma \ref{le:Plclass}. Since $v \leq 2 q \ell \theta$, we obtain, for some new constant $C >0$,
\begin{equation*}
\dE \PAR{ \| B^{\ell} \|^{2\theta} } \leq \rho^{2 \ell\theta }  C^{q\theta} ( (q \ell \theta)^9  n)^{ q} \sum_{r=0}^{\infty} \left( \frac{  ( C q \ell \theta )^8}{ n}\right)^r \sum_{e_1 = 0}^\infty  \left( \frac{    (C q\ell \theta) ^{4 } } {n^{\frac{ 1}{10q}}} \right)^{e_1}.
\end{equation*}
For our choices of parameters in \eqref{eq:choixql}, for all $n$ large enough, 
$$
(C q\ell \theta) ^{4 q} \leq  (2C C_0 q^2 \theta )^{4q} \leq n^{1/2^4 + 4 \delta}.
$$
This last expression is bounded by $n^{1/12}$ for $\delta$ small enough. In particular, the above geometric series in $e_1$ converges and
$$
\sum_{e_1 = 0}^\infty  \left( \frac{    (C q\ell \theta) ^{4 } } {n^{\frac{ 1}{10q}}} \right)^{e_1} \leq \sum_{e_1 = 0}^\infty  n^{-\frac{ e_1 }{60q} } \leq \frac{1}{1 - n^{-\frac{1}{60q}}} \leq  2.
$$
for all $n$ large enough. Similarly, if $\delta$ is small enough, the series in $r$ converges and is bounded by $2$. We fix such choice of $\delta >0$; we deduce that for all $n$ large enough, 
$$
\dE \PAR{ \| B^{\ell} \|^{2\theta} }  \leq 4 \rho^{2 \ell\theta }  C^{q\theta} ( (q \ell \theta)^9  n)^{ q}. 
$$
We fix the constant $C_0$ in \eqref{eq:choixql} such that $C^{q\theta/(2\ell\theta)} \leq C^{1/(2C_0)} \leq 1+\veps$. We note that 
$$
n^{q/(2\ell \theta)} \leq \exp \PAR{\frac{\ln n}{2C_0 n^{\delta/q}}  }. 
$$
If we further assume that $q \leq \max (\delta/2 , 2^{-7}) \ln n / (\ln \ln n)$, we get that $n^\theta \geq n^{\delta / q} \geq (\ln n)^2$.  It follows that for these choices of parameters, for all $n$ large enough, 
$$
4 ( (q \ell \theta)^9  n)^{ q} \leq ( 1+ \veps)^{2\ell \theta}.  
$$
We thus have checked that 
$$
\dE \PAR{ \| B^{\ell} \|^{2\theta} }  \leq \PAR{ \rho (1+\veps)^2  }^{2 \ell \theta}. 
$$
We have $\rho(B_\star) \leq \| B_\star\| \leq C$ for some $ C >0$. Also, from \eqref{eq:defrho}, we have $\rho \leq \rho(B_\star) + \veps$. We get $$\dE \PAR{ \| B^{\ell} \|^{2\theta} } \leq (\rho(B_\star)  + \veps')^{2\ell \theta},$$ 
where $\veps'$ can be taken arbitrarily small if $\veps$ is small enough. From Markov inequality, we obtain that 
$$
\dP \PAR{ \| B^{\ell} \|^{1/\ell} \geq (1+\veps')(\rho(B_\star)  + \veps') }\leq (1+\veps')^{-2\ell\theta}.
$$ 
We easily obtain the required statement by adjusting the value of $\veps$ and the constants. 
\end{proof}

\subsection{Proof of Theorem \ref{th:main}}

The orthogonal projection of $A$ defined in \eqref{eq:defA} onto $H_r^\perp$ can be written as
\begin{equation}\label{eq:def[A]}
[A] = a_0 \otimes 1 + \sum_{i=1}^{2d} a_i \otimes [V_i]. 
\end{equation}
We fix $\veps >0$ and let $\delta = \delta(\veps)$ be as in Theorem \ref{prop:edgeAB}(ii). In view of Theorem \ref{prop:edgeAB}, the event 
$$
\BRA{ \| A_{|H_r^\perp} \| \geq \| A_\star \| + \veps }
$$
is contained in the event 
$$
\cE_\veps  = \bigcup_{a = (a_1,\ldots a_{2d}) \in S^{2d}_\veps} \cE_{\delta}(a) ,
$$
where 
$$
\cE_\delta (a) = \{  \rho(B) \geq \rho(B_\star) + \delta \},
$$ with $B = B(a)$ is an in \eqref{eq:defBV2}, $B_\star(a) = B_\star$ is as in \eqref{eq:defBstar} and 
$$S_\veps = \{ b \in M_r(\dC) : \| b \| \leq \veps^{-1} \}.$$ 
To prove Theorem \ref{th:main} it is thus sufficient to check that for any $\veps >0$, for all $n$ large enough,
\begin{equation}\label{eq:tdb0}
\dP ( \cE_\veps) \leq n^{-2}. 
\end{equation}

To prove \eqref{eq:tdb0}, we need to use a net on $S_\veps^{2d}$. A similar argument appears in \cite{MR4024563}.  
Due to the lack of uniform continuity of spectral radii, we perform the net 
argument with operator norms. From \eqref{eq:defBV2}, for $a \in M_r(\dC)^{2d}$, the map $a \mapsto B(a)$ is linear and
$ \| B(a ) \| \leq (2d-1) \| a \|
$, 
where 
$$
\| a \| = \sum_{i =1}^{2d} \| a_i   \|. 
$$
Let $\ell = \lfloor C q \rfloor$ be as in Theorem \ref{th:FKB}. The map $a \mapsto B^\ell (a)$ satisfies a deviation inequality 
\begin{eqnarray}
\| B^{\ell} (a ) - B^{\ell} (a') \| & \leq &   \ell  \max( \| B (a) \| , \|B (a') \|  )^{\ell-1} \| B (a - a') \| \nonumber \\
&\leq  &\ell (2d-1)^\ell \max( \| a \| , \| a' \| ) ^{\ell-1} \| a - a' \|. \label{eq:LIPB}
\end{eqnarray}

For a given $\eta >0$, the net $N_\eta$ of $S_\veps^{2d}$ is built as follows. First, since all matrix norms are equivalent and 
$M_r( \dC) \simeq \dR^{2r^2}$, there exists a subset $N^1 _{\eta} \subset \{ b \in M_r(\dC) : \| b\| \leq \veps^{-1}\}$ of cardinality 
at most $( c / (\veps \eta) )^{2r^2}$ such that for 
any $b \in S_\veps$, there exists $b' \in N_\eta$ with $\| b - b' \| \leq \eta$ (the constant $c$ depends 
on $r$). We set $N_\eta = (N^1_\eta)^{2d}$.  From \eqref{eq:LIPB}, for some new constant $c >0$ (depending on $\veps,r,d$), for any 
$a \in S_\veps$, there exists $a' \in N_\eta$ such that
$$
\| B^{\ell} (a ) - B^{\ell} (a') \| \leq  \ell (2d-1) \PAR{\frac{ 2d-1}{ \veps}}^{\ell -1}  \eta \leq c^\ell \eta. 
$$

Besides, from \eqref{eq:src}, for all $\eta \leq \eta_0$ small enough, 
\begin{equation}\label{eq:contrhoB}
\left| \rho(B_\star(a)) - \rho(B_\star (a')) \right|\leq \frac \delta  3,
\end{equation}
If $\eta=\min(\eta_0, (\delta / 3 c)^\ell )$ and $\mathcal E_{\delta / 3} (a') $ does not hold, we deduce that
\begin{eqnarray*}
 \| B^{\ell} (a) \| &  \leq &    \| B^{\ell} (a')\| +  \| B^{\ell} (a ) - B^{\ell} (a') \| \\
& < & \PAR{ \rho(B_\star (a')) + \frac \delta 3 }^\ell +  \PAR{\frac \delta 3 }^\ell \\
& < & \PAR{ \rho(B_\star (a')) +\frac {2\delta} 3 }^\ell \\
& < & \PAR{ \rho(B_\star (a)) +\delta }^\ell,
\end{eqnarray*}
where we have used \eqref{eq:contrhoB} at the last line. 
We find that, for our choice of $\eta$,
$$
\cE_\veps = \bigcup_{a \in S_\veps^{2d}} \mathcal E_\delta (a) \subset \bigcup_{a \in N_\eta} \mathcal E _{\delta / 3} (a) ,
$$
and, for some $c_1 >0$ (depending on $\veps$, $r$ and $d$),
\begin{equation*}\label{eq:epsnet}
| N_\eta| \leq c_1 ^\ell .
\end{equation*}
We may now use the union bound to obtain an estimate of \eqref{eq:tdb0}:
\begin{eqnarray*}
\dP \PAR{  \mathcal E_\veps } & \leq & \sum_{a \in N_\eta}  \dP\PAR{  \mathcal E _{\delta / 3} (a)} \\
& \leq & |N_\eta | C \exp ( -(\ln n)^2 ) , 
\end{eqnarray*}
where at the second line, we have used Theorem \ref{th:FKB} (the constant $C$ depends on $\veps,d,r$). For our choice of $\ell$, we have 
$|N_{\eta}|  \leq c_1^\ell \leq n$ for all $n$ large enough. 
The bound \eqref{eq:tdb0} follows.  \qed
 
 \section{Appendix: Effective linearization for unitaries}
\label{sec:appendix}

To prove the strong convergence, we rely on the following variant of a Theorem by Pisier: 
\begin{theorem}[\cite{MR1401692}, Proposition 6]\label{thm-pisier}
Let $A,B$ be two unital $C^*$-algebras generated respectively by unitaries $\{U_i,i\in \INT{d}\}$ and
$\{V_i,i\in \INT{d}\}$. For convenience, define $U_{-i}:=U_i^*$, $V_{-i}:=V_i^*$, $U_0=1_A$, $V_0=1_B$.
Assume that, for any $k\in\mathbb{N}_*$ and for any $a_{-d},\ldots ,a_d\in M_k(\dC)$,
$$||\sum_{i=-d}^da_i\otimes U_i||=||\sum_{i=-d}^da_i\otimes V_i||,$$
then the map $U_i\to V_i$ extends to a unital $*$-isomorphism between $A$ and $B$.
\end{theorem}

In particular, it means that given a non-commutative polynomial in $U_i$, its norm is determined by the collection of
all $\{||\sum_{i=-d}^da_i\otimes U_i||,a_i\in M_k(\dC),i\in\{-d,\ldots, d\}, k\in\mathbb{N}_*\}$. However, the theorem does not make this dependence explicit. This section aims to make this dependence explicit and robust enough to allow concrete estimates for convergence speeds
 for any non-commutative polynomial.

Before, let us outline how we use Theorem \ref{thm-pisier}. We replace $A$ by a sequence of unital $C^*$-algebras $A^{(n)}$ and
$U_i$ by $U_i^{(n)}$. 
Let $\mathbb{C}[\Fd]$ be the $*$-algebra of the free group and $P\in \mathbb{C}[\Fd]$, where the canonical generators of $\Fd$ are $u_1,\ldots , u_d$. The map $u_i\mapsto U_i$ generates a $*$-homomorphism $\mathbb{C}[\Fd]\to A$
which we denote by $P\mapsto P_U$. Likewise, we have a $*$-homomorphism $\mathbb{C}[\Fd]\to A$
which we denote by $P\mapsto P_V$.

If, for any $k\in\mathbb{N}_*$ and for any $a_{-d},\ldots ,a_d\in M_k(\dC)$,
$$\lim_n||\sum_{i=-d}^da_i\otimes U_i^{(n)}||=||\sum_{i=-d}^da_i\otimes V_i||,$$
then
$$\lim_n||P_{U^{(n)}}||=||P_V||.$$
Indeed, if  consider the $C^*$-algebra $A=\prod_{n\ge 1}A^{(n)}/\mathcal{I}$
where $\mathcal{I}$ is the two-sided ideal of sequences of elements in $A^{(n)}$ whose norm tends to zero
as $n\to \infty$, and call $U_i$ the image of $(U_i^{(n)})_{n\ge 1}$ under taking the quotient. 
Then, since finite-dimensional matrix algebras are exact, the $C^*$ subalgebra of $A$ generated by $U_i$ satisfies the hypotheses of Pisier's proposition. Thus, the isomorphism of the conclusion and the definition of $A$ ensures that
$\lim_n||P_{U^{(n)}}||=||P_V||$ for any $P\in \mathbb{C}[\Fd]$.

As mentioned above,
Pisier's theorem is non-constructive. Let us outline the idea of its proof. In the language of operator spaces, the map
$U_i\to V_i$ is a completely contractive map from the operator space $\text{span} (U_i)$ to $\text{span} (V_i)$.
Therefore, the Arveson-Wittstock extension theorem applies and yields an unital completely contractive map from 
$A$ to $B$. The classification of completely contractive maps from $A$ to $B$ (Stinespring theorem) and the fact
that the generators are unitaries implies that this complete contraction must be an isometric homomorphism, 
which concludes the proof.

Let us digress and compare with Haagerup and Thorbj{\o}rnsen's result in \cite{MR2183281}, where the first statement of a linearization argument can be found for selfadjoint elements. Two remarks are in order. As for Pisier's result, the initial version is not constructive. 
A constructive version has been found later, see, e.g., \cite{MR3585560}, Chapter 10 for a comprehensive explanation.
Interestingly, Pisier's assumptions are simpler because only an equality on the norm of linear pencils is required. In contrast, for Haagerup and Thorbj{\o}rnsen, an equality of spectrum is needed (which looks stronger at first sight). 
As far as we can tell, a constructive proof of Pisier has not been given yet. We fill in this gap by providing a linearization. It is to be noted that this linearization does not require knowledge about the inverse of matrices.

The first lemma proves that it is enough to restrict to selfadjoint pencils and that the result also extends to 
all rectangular matrix coefficients.
\begin{lemma}\label{lem-selfadjoint-is-enough}
The following are equivalent:
\begin{enumerate}
\item For any $a_{-d},\ldots ,a_d\in M_k(\dC)$,
$$||\sum_{i=-d}^da_i\otimes U_i||=||\sum_{i=-d}^da_i\otimes V_i||.$$
\item For any $a_{-d},\ldots ,a_d\in M_{N,k}(\dC)$,
$$||\sum_{i=-d}^da_i\otimes U_i||=||\sum_{i=-d}^da_i\otimes V_i||.$$
\item For any $a_{-d},\ldots ,a_d\in M_k(\dC)$, $a_{-d},\ldots ,a_d\in M_k(\dC)$, $a_i=a_{-i}^*$.
$$||\sum_{i=-d}^da_i\otimes U_i||=||\sum_{i=-d}^da_i\otimes V_i||.$$
\end{enumerate}
\end{lemma}

\begin{proof}
The equivalence between the first point and the second point follows from the fact that padding a rectangular matrix
 with zeros into a square matrix does not change its operator norm.
The equivalence between the first point and the third point follows from the fact that for any operator $A: L(H_1,H_2)$,
$$||A||=
||\begin{pmatrix}
0 & A \\
A^* & 0
\end{pmatrix}||.$$
The conclusion follows.
\end{proof}

The second lemma proves that it is enough to restrict to positive pencils. 
\begin{lemma}\label{lem-positive-is-enough}
The following are equivalent:
\begin{enumerate}
\item For any $a_{-d},\ldots ,a_d\in M_k(\dC)$,
$$||\sum_{i=-d}^da_i\otimes U_i||=||\sum_{i=-d}^da_i\otimes V_i||.$$
\item For any $a_{-d},\ldots ,a_d\in M_k(\dC)$ such that 
$\sum_{i=-d}^da_i\otimes u_i\ge 0$ is positive as an element of $\mathbb{C}[\Fd]$ (i.e. a finite sum of selfadjoint squares), then
$$||\sum_{i=-d}^da_i\otimes U_i||=||\sum_{i=-d}^da_i\otimes V_i||.$$
\end{enumerate}
\end{lemma}

\begin{proof}
The above follows from the following  fact from spectral theory:
 for a selfadjoint element $P$ in a unital $C^*$-algebra $A$,
for any $c>0$, if one knows $||x1+P||$ for any $x\in \mathbb{R}, |x|\ge c$, then one knows 
$||P||$. This concludes the proof, along with the fact that $||P||=||-P||$.
\end{proof}

We also need the following crucial technical lemma
\begin{lemma}\label{lem-specific-to-unitaries}
Let $G\subset F_d$ be a finite, symmetric subset of the free group (i.e., $G^{-1}=G$). Consider the set
$G^2=\{gh, g\in G, h\in G\}$. 
Let $Q=\sum_{g\in G^2}a_g\otimes \lambda_g \in M_k(\dC )\otimes \mathbb{C}[\Fd]$ be a self-adjoint element (equivalently, 
$a_g^*=a_{g^{-1}}$).

Then, there exists $N\ge k$ such that for all $c >0$ large enough, there exists $P = \sum_{g \in G} b_g\otimes \lambda_g \in M_{N,k}(\dC )\otimes \mathbb{C}[\Fd]$ which satisfies
$P^*P=Q+c1$.
\end{lemma}

\begin{proof}
Given $Q=\sum_{g\in G^2}a_g\otimes \lambda_g \in M_k(\dC )\otimes C[\Fd]$, pick
a selfadjoint element
$\tilde Q\in  \mathrm{End} (\underbrace{\mathbb{C}^k\oplus \cdots\oplus\mathbb{C}^k}_{\text{$|G|$ times}})$
 selfadjoint
that satisfies the following property: writing $\tilde Q=(\tilde Q_{g,h})_{g,h\in G}$, 
$\sum_{g,h\in G, g^{-1} h=r}\tilde Q_{g,h}=a_r$.
Such a $\tilde Q$ can always be found
(for example, take  $\tilde Q_{g,h}= a_{g^{-1} h} / |\{(g',h')\in G^2, (g')^{-1}h'=g^{-1} h\}|$). Let $c>0$ be large enough to ensure that $\tilde Q+ c1$ is positive and let
 $\tilde P$ be a selfadjoint square root of $\tilde Q+ c1$.
 Write $\tilde P= (\cdots |\tilde P_g |\cdots )_{g\in G}$ where $\tilde P_g\in M_{k|G|,k}(\dC ) $ is a ``block column vector'', i.e. 
 an operator $\mathbb{C}^k\to \underbrace{\mathbb{C}^k\oplus \cdots\oplus\mathbb{C}^k}_{\text{$|G|$ times}}$.
 Then, the operator $P=\sum_{g\in G} \tilde P_g\otimes \lambda_g$ in $M_k(\dC )\otimes \mathbb{C}[\Fd]$ satisfies
 $$P^*P= \sum_{g,h  \in G} \tilde P_g^* \tilde P_h \otimes \lambda(g^{-1} h) = \sum_{g,h  \in G} ( \tilde Q_{g,h} + c \IND_{g = h} \cdot 1_k) \otimes \lambda(g^{-1} h)  = Q + c  |G| 1,$$
 and this concludes the proof. 
 \end{proof}
Following notations of operator systems, we call
$$S_k(U)=\text{span} \{\sum_{i=-d}^da_i\otimes U_i, a_i\in M_k(\dC)\},$$
and $S_k(U)^{(l)}$ the vector space spanned by the product of all $l$-tuples of elements in $S_k(U)$.
By definition, $S_k(U)=S_k(U)^{(1)}$. In addition, one checks that $S_k(U)^{(l)}\subset S_k(U)^{(l+1)}$ and
$\cup_{l\ge 1} S_k(U)^{(l)}$ is $M_k(\dC)$ tensored by the $*$-algebra generated by $U_i$.

For what follows, we extend the homomorphism $\mathbb{C}[F_d]\to A$ by tensorising by matrices, and we keep the same notation 
$P\mapsto P_U$ (resp. $P\mapsto P_V$ for the homomorphism $M_k(\dC )\otimes\mathbb{C}[\Fd]\to M_k(\dC )\otimes A$).
The following lemma concludes the proof
\begin{lemma}
Consider the map $S_k(U)^{(l)}\subset M_k(\dC )\otimes A$ to $S_k(V)^{(l)}\subset M_k(\dC )\otimes B$ 
defined as follows: for $\tilde Q\in S_k(U)^{(l)}$, pick an
element $Q=\sum_{g\in G^{2l}}a_g\otimes \lambda_g \in M_k(\dC )\otimes \mathbb{C}[\Fd]$ such that
$Q_U=\tilde Q$. Then, the image of  $\tilde Q$ is $Q_V$. Assume that this map is well-defined, linear, and
 isometric for all $k$, then the same result holds if one replaces $l$ by $2l$.
\end{lemma}

\begin{proof}
Take $G=\{u_{-d},\ldots  \ldots, u_d\}$ where $u_1, \ldots, u_d$ are generators 
of the free group $F_d$ and $u_{-i}=u_i^{-1}, u_0=e$.
Consider a self-adjoint element $Q=\sum_{g\in G^{2l}}a_g\otimes \lambda_g \in M_k(\dC )\otimes \mathbb{C}[\Fd]$.
Thanks to Lemma \ref{lem-selfadjoint-is-enough} (point 3), it will be enough to prove that $||Q_U||=||Q_V||$.

Given $Q\in M_k(\dC )\otimes\mathbb{C}[\Fd]$ its images are $Q_U\in A$ and $Q_V\in B$.
Then by Lemma \ref{lem-specific-to-unitaries}, there exists an integer $N$ such that for all $c$ large enough,
$$Q+c1=P^*P,$$ where $P=\sum_{g\in G^{l}}b_k \otimes \lambda_g 
\in M_{N,k}(\mathbb{C})\otimes \mathbb{C}[\Fd]$.
Then, by Lemma \ref{lem-selfadjoint-is-enough} (point 2) we have $||P_U||=||P_V||$. The same is true for $-Q$. Therefore, by the $C^*$ norm axiom,
$||Q_U+c1_A||=||Q_V+c1_B||$ for any $|c|$ large enough. Thanks to Lemma \ref{lem-positive-is-enough}, this
implies that $||Q_U||=||Q_V||$ and we have extended the isometry.

To conclude, the representations $\mathbb{C}[\Fd]\to A$ and $\mathbb{C}[\Fd]\to B$, given respectively by $Q\mapsto Q_U$ and $Q\mapsto Q_V$
yield the $C^*$-norm on $\mathbb{C}[\Fd]$, therefore the map $U_i\mapsto V_i$ extends to an isomorphism $A\to B$. \end{proof}

Two remarks are in order. Firstly, the above argument is specific to unitaries. Indeed, if one replaces unitaries with selfadjoint elements,
the proof of the lemma \ref{lem-specific-to-unitaries} breaks down. For example, it does not a priori allow to construct
$X_1^2-X_2^2 +c1$ as a square of polynomials. 

Secondly,  by following the same reasoning, the lemma can be quantitatively improved as follows. If instead
of an equality, assume we have an inequality of the following type: there exists $\varepsilon_n >0$ tending to zero as $n\to\infty$ such that 
$||P_{U^{(n)}}||\le ||P_V||(1+\varepsilon_n)$. Then, 
$||Q_U+c1_A|| \le ||Q_V+c1_B||(1+\varepsilon_n)^2$.
Since an upper bound on an appropriate $c$ can be given as a function of the norm of each matrix coefficient
attached to a group element, there exists a constant $C$ that depends explicitly on 
$P\in M_k(\dC )\otimes\mathbb{C}[\Fd]$ (and on nothing else)
such that for $n$ large enough,
$$||Q_{U^{(n)}}|| \le ||Q_V||(1+C \varepsilon_n).$$
Naturally, an inequality in the opposite direction can be obtained similarly.
Iterating, we obtain the following
\begin{proposition}
If there exists a sequence tending to zero $\varepsilon_n$ such that
 any $k\in\mathbb{N}_*$ and for any $a_{-d},\ldots ,a_d\in M_k(\dC)$,
$$||\sum_{i=-d}^da_i\otimes U_i^{(n)}||-||\sum_{i=-d}^da_i\otimes V_i||=O(\varepsilon_n),$$
then the same holds at the level of polynomials:
$$||P_{U^{(n)}}||-||P_V||=O(\varepsilon_n).$$
\end{proposition}
In addition, uniform bounds in the $a_i$'s on the speed of convergence of $||\sum_{i=-d}^da_i\otimes U_i^{(n)}||$ are
not needed, and the final estimate $||P_{U^{(n)}}||-||P_V||=O(\varepsilon_n)$ can be made explicit as a function of $P$ and the $a_i$'s.

\bibliographystyle{abbrv}
\bibliography{bib}

\end{document}